\documentclass[a4,12pt]{article}

\usepackage{amsfonts, amsmath, amssymb, amsgen, amsthm, amscd,latexsym,mathrsfs}

\usepackage{color}
\usepackage[all]{xy}

\usepackage[top=3.2cm, left=2.2cm, bottom=2.2cm, right=2.2cm]{geometry}

\def\Diff{\mathop{\rm Diff}\nolimits}
\def\End{\mathop{\rm End}\nolimits}

\def\Id{\mathop{\rm Id}\nolimits}

\def\Ad{\mathop{\rm Ad}\nolimits}
\def\ad{\mathop{\rm ad}\nolimits}

\def\Tr{\mathop{\rm Tr}\nolimits}

\def\Cb{{\mathbb C}}

\def\Nb{{\mathbb N}}
\def\Rb{{\mathbb R}}

\def\Fc{{\cal F}}

\def\Hc{{\cal H}}

\def\Ic{{\cal I}}

\def\Uc{{\cal U}}

\def\Cc{{\cal C}}

\def\a{\alpha}
\def\b{\beta}
\def\d{\delta}
\def\D{\Delta}
\def\g{\gamma}

\def\lb{\lambda}

\def\s{\sigma}

\def\ve{\varepsilon}
\def\vp{\varphi}

\def\0b{\bf 0}

\def\nb{\nabla}
\def\ot{\otimes}

\def\ra{\rightarrow}

\def\rt{\triangleright}

\def\lt{\triangleleft}

\def\acl{\blacktriangleright\hspace{-4pt}\vartriangleleft }

\def\topcl{\hat{\blacktriangleright\hspace{-4pt} < }_\pi}
\def\topal{\hat{>\hspace{-4pt}\vartriangleleft}_\pi}
\def\tacl{\hat{\blacktriangleright\hspace{-4pt}\vartriangleleft }_\pi}

\def\bi{\bowtie}

\def\p{\partial}

\def\0D{\Delta^{(0)}}
\def\1D{\Delta^{(1)}}
\def\Db{\blacktriangledown}

\def\wg{\wedge}

\newcommand{\Fg}{\mathfrak{g}}

\newcommand{\Fk}{\mathfrak{k}}

\def\projot{\widehat{\ot}_\pi}
\def\projhat{\widehat{\ot}}
\def\oprojdots{\projot\cdots\projot}

\newtheorem{theorem}{Theorem}[section]
\newtheorem{remark}[theorem]{Remark}
\newtheorem{proposition}[theorem]{Proposition}
\newtheorem{lemma}[theorem]{Lemma}
\newtheorem{corollary}[theorem]{Corollary}

\newtheorem{example}[theorem]{Example}
\newtheorem{definition}[theorem]{Definition}

\def\ni{\noindent}

\def\build#1_#2^#3{\mathrel{
\mathop{\kern 0pt#1}\limits_{#2}^{#3}}}
\newcommand{\ps}[1]{~\hspace{-4pt}_{^{(#1)}}}

\newcommand{\ns}[1]{~\hspace{-4pt}_{_{{<#1>}}}}

\newcommand{\nsb}[1]{~\hspace{-4pt}_{^{[#1]}}}

\def\wdots{\wedge\dots\wedge}

\def\one{{\bf 1}}

\newcommand{\ie}{{\it i.e.\/}\ }

\def\a{\alpha}
\def\b{\beta}
\def\d{\delta}

\def\g{\gamma}

\def\lb{\lambda}

\def\s{\sigma}

\def\ve{\varepsilon}

\def\vp{\varphi}

\def\D{\Delta}

\def\Lb{\Lambda}

\def\nb{\nabla}

\def\ot{\otimes}
\def\part{\partial}

\def\ra{\rightarrow}

\def\text{\hbox}

\def\nb{\nabla}
\def\ot{\otimes}

\def\ra{\rightarrow}

\def\Ad{\mathop{\rm Ad}\nolimits}

\def\Diff{\mathop{\rm Diff}\nolimits}

\def\End{\mathop{\rm End}\nolimits}

\def\Id{\mathop{\rm Id}\nolimits}
\def\exp{\mathop{\rm exp}\nolimits}

\def\lra{\longrightarrow}

\def\build#1_#2^#3{\mathrel{
\mathop{\kern 0pt#1}\limits_{#2}^{#3}}}

\numberwithin{equation}{section}
\begin{document}

\title{\bf  Topological Hopf algebras and their Hopf-cyclic cohomology}
\author{
\begin{tabular}{cc}
Bahram Rangipour \thanks{Department of Mathematics  and   Statistics,
     University of New Brunswick, Fredericton, NB, Canada    \quad  Email: bahram@unb.ca, }
     \quad and  \quad  Serkan S\"utl\"u \thanks{Department of Mathematics,
     I\c{s}{\i}k University, \.Istanbul, Turkey  \quad Email: serkan.sutlu@isikun.edu.tr }
      \end{tabular}  }

\date{}
\maketitle

\abstract{\ni  A natural extension of the Hopf-cyclic cohomology, with coefficients, is introduced  to  encompass topological Hopf algebras. The topological theory allows to work with infinite dimensional Lie algebras. Furthermore, the category of coefficients (AYD modules) over a topological Lie algebra and those over its universal enveloping (Hopf) algebra are isomorphic. For topological Hopf algebras, the category of coefficients is identified with the representation category of a topological algebra called the anti-Drinfeld double. Finally, a topological van Est type isomorphism is detailed, connecting the Hopf-cyclic cohomology to the relative Lie algebra cohomology with respect  to a maximal compact subalgebra.}\\



\newpage

\tableofcontents

\section{Introduction}

In the fundamental  work of Connes and Moscovici on their local index formula \cite{ConnMosc98}, Hopf-cyclic cohomology emerged naturally as a cohomology theory on the domain of a characteristic homomorphism (whose range being the cyclic cohomology of an algebra on which the Hopf-algebra acts). As such, this cyclic cohomology theory for Hopf algebras was designed as a computational gadget to access the K-theory invariants. 

\ni Moreover, it was also \cite{ConnMosc98} in which a Hopf algebra $\Hc_n$ - called the Connes-Moscovici Hopf algebra - is introduced, and the Hopf-cyclic cohomology of $\Hc_n$ is identified with the Gelfand-Fuks cohomology of the (infinite dimensional) Lie algebra $W_n$ of formal vector fields on $\Rb^n$. That is, from the very beginning, the Hopf-cyclic cohomology is intimitely related to the theory of the characteristic classes of foliations.

\ni However, although feasible techniques (employing the bicrossed product structure) are introduced later to compute the Hopf-cyclic cohomology of $\Hc_n$, \cite{MoscRang09}, there appeared no direct passage between the Hopf-cyclic cohomology of $\Hc_n$, and the Gelfand-Fuks cohomology of $W_n$.

\ni It is this missing bridge, between the characteristic classes of foliations and the Hopf-cyclic cohomology, we strive to develop; via which one may attempt to calculate the characteristic classes using the Hopf-cyclic cohomology techniques.

\ni The first step towards this direction was taken in \cite{RangSutl-II}, where we have introduced a cyclic (co)homology theory for Lie algebras. We then identified this Lie-cyclic cohomology of a Lie algebra with the Hopf-cyclic cohomology of its universal enveloping (Hopf) algebra, where both cohomologies were allowed to bear the coefficients; the stable-anti-Yetter-Drinfeld (SAYD) modules \cite{HajaKhalRangSomm04-II,HajaKhalRangSomm04-I,Kayg05,JaraStef06}.

\ni The key step of this identification was the construction of a coaction of the universal enveloping algebra on the coefficient space, given the coaction of the Lie algebra. What was observed in \cite[Prop. 5.7]{RangSutl-II}, see also \cite[Thm. 4.4]{RangSutl-III}, was that this is possible only if the Lie algebra coaction is ``locally conilpotent'', \cite[Def. 5.4]{RangSutl-II}. The terminology follows from the fact that it encodes (under the action-coaction duality) the nilpotent action of the symmetric algebra $S(\Fg^\ast)$. That was the reason why the truncated Weil algebra could be identified with a Hopf-cyclic complex (the coaction of a Lie algebra $\Fg$ on the space $S(\Fg^\ast)_{[n]}$ of truncated polynomials is locally conilpotent), while this was no longer the case for the Weil algebra without truncation (the coaction of the Lie algebra $\Fg$ on the full symmetric algebra $S(\Fg^\ast)$ is not locally conilpotent, \cite[Ex. 5.6]{RangSutl-II}).

\ni In addition to the local conilpotency restriction, all van Est type isomorphisms we have developed through \cite[Thm. 4.6]{RangSutl-III} and \cite[Thm. 4.10]{RangSutl}, or in particular, \cite[Coroll. 4.11 - 4.13]{RangSutl}, worked with finite dimensional Lie algebras. As such, the infinite dimensional Lie algebra $W_n$ of formal vector fields, its Gelfand-Fuks cohomology, and hence the characteristic classes were beyond the reach of the Hopf-cyclic theory.

\ni In the present paper, we achieve to extend the Hopf-cyclic (as well as the Lie-cyclic) theory to a level that allows to work with infinite dimensional Lie algebras, as well as the exponentiation of the Lie algebra coactions, in the presence of a suitable topology. 

\ni More precisely, we obtain the following result.
\\\\
{\bf Proposition \ref{prop-g-AYD-to-Ug-AYD}.}
{\it Let $\Fg$ be a topological Lie algebra, and $V$ a topological $\Fg$-module / comodule. Then $V$ is a topological right-left AYD module over $\Fg$ if and only if it is a topological right-left AYD module over $U(\Fg)$.}
\\\\
\ni Furthermore, along the lines of the van Est type isomorphisms \cite[Thm. 4.6]{RangSutl-III},  \cite[Prop. 7]{ConnMosc98}, \cite[Thm. 15]{ConnMosc}, \cite[Thm. 4.10]{RangSutl}, \cite[Thm. 6.2]{RangSutl-II}, and \cite[Thm. 4.6]{RangSutl-III}, we obtain the following.
\\\\
{\bf Theorem \ref{thm-main-identf}.}
{\it Let $(G_1,G_2)$ be a matched pair of Lie groups, with Lie algebras $(\Fg_1,\Fg_2)$. Then the periodic Hopf-cyclic cohomology of the Hopf algebra $\Fc^\infty(G_2)\tacl U(\Fg_1)$  with coefficients in the canonical MPI ${}^\s k_\d$ is isomorphic with the Lie algebra  cohomology, with trivial coefficients, of $\Fg_1\bi\Fg_2$ relative to the maximal compact Lie subalgebra $\Fk$ of $\Fg_2$. In short,
\begin{equation*}
HP^\ast_{top}(\Fc^\infty(G_2)\tacl U(\Fg_1),{}^\s k_\d) \cong \widetilde{HP}^\ast(\Fg_1\bi\Fg_2,\Fk).
\end{equation*}}

\ni Recently, the topological Hopf-cyclic theory developed here has been successfully applied to the construction of a direct link between the characteristic classes of foliations and the Hopf-cyclic cohomology that allows the transfer of the cocycles. The (topological) Hopf-cyclic cohomology of the Connes-Moscovici Hopf algebra $\Hc_n$, with infinite dimensional coefficients (the space of formal differential 1-forms) was identified with the Gelfand-Fuks cohomology of the Lie algebra $W_n$ of formal vector fields on $\Rb^n$ (with the same infinite dimensional coefficient space), \cite{RangSutlYazd17-arxiv} . More precisely, we refer the reader to \cite[Sect. 5]{RangSutlYazd17-arxiv} for the illustration of the transfer of the classes in the case $n=1$.
 


\ni The plan of the paper is as follows. Section \ref{sect-top-Lie-Hopf} is devoted to the straightforward extension of the Lie-Hopf algebras of \cite{RangSutl-I-arxiv} to the realm of topological Hopf algebras. Upon recalling the topological Hopf algebras in Subsection \ref{subsect-top-Hopf-alg}, we discuss the bicrossproducts of topological Hopf algebras in Subsection \ref{subsect-Lie-Hopf-bicross}, in order to develop our main example; the Hopf algebra $\Fc^\infty(G_2)\,\tacl\,U(\Fg_1)$. Section \ref{sect-Hopf-cyclic-top} is devoted to a carefull generalization of the Hopf-cyclic cohomology theory of Hopf algebras, to the level of topological Hopf algebras. More precisely, following the chronological order, in Subsection \ref{subsect-Hopf-cyclic-comp} we develop first the canonical modular pair in involution (MPI) for a topological Lie-Hopf algebra towards the Hopf-cyclic cohomology with trivial coefficients. This is important as one of our main tasks is to identify the Hopf-cyclic cohomology of the Hopf algebra $\Fc^\infty(G_2)\tacl U(\Fg_1)$ with coefficients that come from the canonical MPI over it. We then carefully upgrade the usual Hopf-cyclic complex with general SAYD coefficients to the level of the topological Hopf algebras. In Subsection \ref{subsect-Hopf-cyclic-coeff} we achive yet another important result. More explicitly, we characterise the category of SAYD modules over a Hopf algebra $\Hc$ as the representation category of the algebra $\Hc^\circ\projot \Hc$, called the anti-Drinfeld-double. This generalises \cite[Prop. 4.2]{HajaKhalRangSomm04-I}, which holds only for the finite dimensional Hopf algebras. Another generalisation to the topological world is discussed in Section \ref{sect-cycl-cohom-top-Lie}, this time from the finite dimensional Lie algebras to topological Lie algebras (which are now allowed to be infinite dimensional). To this end, Subsection \ref{subsect-comod-lie-alg} provides a study of the category of SAYD modules over the universal enveloping (Hopf) algebras. It is this subsection that we derive one of our main results; namely Proposition \ref{prop-g-AYD-to-Ug-AYD}. On the next subsection, Subsection \ref{subsect-cyclic-complexes-lie-alg}, we present a rather straightforward generalisation of \cite{RangSutl-II}, the theory of cyclic cohomology (with coefficients) of Lie algebras, to topological Lie algebras (of possibly infinite dimension). Finally, in Section \ref{sect-comput}, more precisely in Subsection \ref{subsect-comput-Lie-Hopf} we identify the Hopf-cyclic cohomology of the Hopf algebra $\Fc^\infty(G_2)\tacl U(\Fg_1)$ with the cyclic cohomology of the Lie algebra $\Fg_1 \bowtie \Fg_2$. A critial step of this identification is the isomorphism between the (coalgebra) Hochschild cohomology of the topological $\Fg_1$-Hopf algebra $\Fc^\infty(G_2)$ and the differentiable group cohomology of $G_2$. This is discussed, in detail, in Subsection \ref{subsect-comput-Hochschild-cohom}.
  
\ni Unless stated otherwise, throughout the text a vector space is assumed to be ``well-behaved'', \cite{BonnFlatGersPinc94,BonnSter05}, that is, either nuclear and Fr\'echet, or nuclear and dual\footnote[1]{Strong dual, \cite[Sect. II.2.3]{Schaefer-book}} of Fr\'echet, \cite{Grot55,Treves-book}. It is explained in \cite[Appdx. 2]{BonnFlatGersPinc94} that any countable dimensional vector space may be naturally endowed with this topology, called the ``natural topology'' therein. We use the projective tensor product $\ot_\pi$, and its completion $\projot$ for the tensor product of topological vector spaces (t.v.s.). For details we refer the reader to \cite{Schaefer-book,Treves-book}. On the other hand, $\Fg$ denotes a countable dimensional Lie algebra, $k$ denotes the ground field (of characteristic zero, equipped with the natural topology mentioned above), and $\ot$ refers to the tensor product over $k$.

\bigskip

\ni  B.R. would like to thank the Hausdorff Institute  in Bonn for its hospitality and support during the time this work was in progress.

\section{The Hopf algebra $\Fc^\infty(G_2)\,\tacl\,U(\Fg_1)$}\label{sect-top-Lie-Hopf}

We shall first recall the topological Hopf algebras, as well as their (co)actions, in Subsection \ref{subsect-top-Hopf-alg}. The examples discussed in Subsection \ref{subsect-top-Hopf-alg} form the building blocks of the bicrossproduct Lie-Hopf algebras of Subsection \ref{subsect-Lie-Hopf-bicross}.

\subsection{Topological Hopf algebras}\label{subsect-top-Hopf-alg}

Let us first recall the notion of a topological Hopf algebra from \cite[Def. 1.2]{BonnFlatGersPinc94}, see also \cite[Def. 2]{BonnSter05}, in which they are defined as well-behaved Hopf algebras.

\begin{definition}
A Hopf algebra (resp. algebra, coalgebra) $H$ whose underlying vector space is a t.v.s. is called a topological Hopf algebra (resp. algebra, coalgebra) if the Hopf algebra structure maps (resp. algebra, coalgebra structure maps) are continuous.
\end{definition}

\ni Any countable dimensional Hopf algebra, equipped with the strict inductive limit topology \cite[Sect. I.13]{Treves-book}, is a topological Hopf algebra by \cite[Prop. 1.5.1]{BonnFlatGersPinc94}. We list below more specific examples.

\begin{example}
{\rm 
The universal enveloping algebra $U(\Fg)$ of a (countable dimensonal) Lie algebra $\Fg$. 
}
\end{example}

\begin{example}
{\rm 
The Hopf algebra $R(G)$ of representative functions on a compact connected Lie group $G$, \cite[Sect. 2.4]{BonnFlatGersPinc94}, see also \cite{BonnSter05} for a linear or semi-simple Lie group $G$.
}
\end{example}


\ni For a further example, the Hopf algebra of (infinitely) differentiable functions on a real analytic group, we adopt the terminology of \cite{HochMost62}.

\begin{definition}\label{def-diffble-map}
Let $G$ be a real analytic group, and let $\Fg$ be its Lie algebra. Let also $V$ be a t.v.s. Then, a continuous map $\vp:G\lra V$ is called differentiable if
\begin{itemize}
\item[(i)] for any $g\in G$, $X\in \Fg$, and $t\in \Rb$,
\begin{equation*}
\vp(g\exp(tX)) = \vp(g)+t\widetilde{\vp}(g,X,t),
\end{equation*}
for a continuous map $\widetilde{\vp}:G\times \Fg\times \Rb\lra V$,
\item[(ii)] for any $X\in \Fg$, the map $X(\vp):G\lra V$ given by
\begin{equation*}
X(\vp)(g):= \widetilde{\vp}(g,X,0)
\end{equation*}
satisfies $(i)$, as well as the map $Y(X(\vp))$, for any $X,Y \in \Fg$, etc.
\end{itemize}
\end{definition}

\begin{example}
{\rm 
Let $G$ be a real analytic group, and $\Fc^\infty(G)$ the space of real valued differentiable functions on $G$. The space $\Fc^\infty(G)$, equipped with the topology of uniform convergence on compact subsets of the functions and of their derivatives, is Fr\'echet \cite[Sect. I.10]{Treves-book}, and nuclear \cite[Sect. III.51]{Treves-book}. As a result, $\Fc^\infty(G)$ is a well-behaved Hopf algebra. See also \cite[Ex. 2.2.1]{BonnSter05}, and \cite[Ex. 1.7]{BonnFlatGersPinc94}.
}
\end{example}

\ni We next record below the definitions of topological modules and comodules, see \cite{Tayl72}.

\begin{definition}
Given a topological algebra $A$, a right topological $A$-module is a t.v.s. $M$ which is a (unital) right $A$-module structure $M\times A \lra M$ that extends to a continuous linear map $M\,\projot\, A \lra M$.

\ni Similarly, given a topological coalgebra $C$, a left topological $C$-comodule is a t.v.s. $V$ which is a (counital) left $C$-comodule $V\lra C \,\projot\, V$.
\end{definition}

\ni An immediate example, that we note here for the later use, is the symmetric algebra $S(\Fg^\ast)$ being a topological $U(\Fg)$-module via the coadjoint action.

\ni Let us next recall the notion of a differentiable $G$-module from \cite{HochMost62}.

\begin{definition}\label{def-diffble-module}
A topological $G$-module $V$ is called a differentiable $G$-module if the map
\begin{equation*}
\rho_v:G\lra V,\qquad \rho_v(g):=v\cdot g
\end{equation*}
is differentiable for any $v\in V$.
\end{definition}



\ni We then have the following characterization.

\begin{proposition}\label{prop-diff-ble-G-module}
Let $G$ be a real analytic group, and $V$ a t.v.s. If $V$ is a left $\Fc^\infty(G)$-comodule, via $v\mapsto v^{\ns{-1}}\,\projot\,v^{\ns{0}} \in \Fc^\infty(G)\,\projot\,V$, then $V$ is a differentiable right $G$-module by
\begin{equation}\label{aux-G-action}
v\cdot g := v^{\ns{-1}}(g)v^{\ns{0}},\qquad \forall\,v\in V,\,g\in G.
\end{equation}
Conversely, if $V$ is a locally finite differentiable right $G$-module\footnote[1]{Any element of $V$ is contained in a finite dimensional (differentiable) $G$-module, \cite[Sect. 1.2]{Hoch-book}.}, then $V$ is a left $\Fc^\infty(G)$-comodule by
\begin{equation}\label{aux-HG-coact}
\nb:V \lra \Fc^{\infty}(G)\,\projot\, V,\quad \nb(v)(g) := v\cdot g.
\end{equation}
\end{proposition}

\begin{proof}
Let $V$ be a left $\Fc^\infty(G)$-comodule. We first show that \eqref{aux-G-action} indeed defines an action. To this end, it is enough to observe that
\begin{align}\label{aux-G-mod-HG-comod}
\begin{split}
&v\cdot (gg') =  v^{\ns{-1}}(gg')v^{\ns{0}}= \D(v^{\ns{-1}})(g,g')\,\projot \,v^{\ns{0}} = \\
& \left(v^{\ns{-1}}\,\projot\,v^{\ns{0}\ns{-1}}\right)(g,g')v^{\ns{0}\ns{0}} = \\
& v^{\ns{-1}}(g)v^{\ns{0}\ns{-1}}(g')v^{\ns{0}\ns{0}} = (v\cdot g)\cdot g'.
\end{split}
\end{align}

\ni Let us now show that the action \eqref{aux-G-action} is a differentiable $G$-action. For any fixed $v\in V$, we show that the map $\rho_v:G \lra V$, given by $\rho_v(g) = v\cdot g$, is differentiable. For any $X\in \Fg$ we have
\begin{align*}
\begin{split}
& \rho_v(g\exp(tX)) = v^{\ns{-1}}(g\exp(tX))v^{\ns{0}} = \\
& v^{\ns{-1}}(g)v^{\ns{0}} + t\widetilde{v^{\ns{-1}}}(g,X,t)v^{\ns{0}} = \rho_v(g) + t\widetilde{\rho_v}(g,X,t),
\end{split}
\end{align*}
where
\begin{equation*}
\widetilde{\rho_v}(g,X,t) := \widetilde{v^{\ns{-1}}}(g,X,t)v^{\ns{0}}.
\end{equation*}
This observation ensures the first condition of Definition \ref{def-diffble-map}. Next we observe for any $Y\in \Fg$ that
\begin{align*}
\begin{split}
& Y(\rho_v)(g\exp(tX)) = \widetilde{\rho_v}(g\exp(tX),Y,0) = \\
& \widetilde{v^{\ns{-1}}}(g\exp(tX),Y,0)v^{\ns{0}}  = Y(v^{\ns{-1}})(g\exp(tX))v^{\ns{0}} = \\
& Y(v^{\ns{-1}})(g)v^{\ns{0}} + t \widetilde{Y(v^{\ns{-1}})}(g,X,t)v^{\ns{0}} = \\
& Y(\rho_v)(g) + t \widetilde{Y(v^{\ns{-1}})}(g,X,t)v^{\ns{0}}.
\end{split}
\end{align*}
Hence, the second requirement of Definition \ref{def-diffble-map} also holds for the map $\rho_v:G \lra V$. We conclude that the $G$-action is differentiable.

\ni Conversely, let $V$ be a locally finite differentiable right $G$-module, and let $v \in V$ be fixed. Then, given $g\in G$, there are $\a^i(g) \in k$, and a basis $v_i \in V$ of the finite dimensional $G$-module containing $v\in V$, so that $v\cdot g = \a^i(g)v_i$. We shall show that $\a^i \in \Fc^{\infty}(G)$. Indeed, for any $X \in \Fg$,
\[
\a^i(g\exp(tX))v_i = \rho_v(g\exp(tX)) = \rho_v(g) + t\widetilde{\rho_v}(g,X,t),
\]
and setting 
\[
\widetilde{\rho_v}(g,X,t) = \widetilde{\a^i}(g,X,t)v_i,
\]
we see that
\[
\a^i(g\exp(tX))v_i = \a^i(g)v_i + t\widetilde{\a^i}(g,X,t)v_i,
\]
that is,
\[
\a^i(g\exp(tX)) = \a^i(g) + t\widetilde{\a^i}(g,X,t).
\]
Let us note also that the mapping $\widetilde{\a^i}:G\times \Fg \times \Rb \to \Rb$ is continuous, being the composition of two continuous mappings $\widetilde{\rho_v}:G\times \Fg \times \Rb \to V$ and $proj_i:V \to \Rb$ given by $\b^jv_j \mapsto \b^i$. Similarly, for any $X,Y \in \Fg$,
\begin{align*}
& Y(\a^i)(g\exp(tX))v_i = \widetilde{\a^i}(g\exp(tX),Y,0)v_i = \widetilde{\rho_v}(g\exp(tX),Y,0) = \\
&  Y(\rho_v)(g\exp(tX)) = Y(\rho_v)(g) + t \widetilde{Y(\rho_v)}(g,X,t)  = \\
&  \widetilde{\rho_v}(g,Y,0)+ t \widetilde{Y(\rho_v)}(g,X,t) = Y(\a^i)(g)v_i + t \widetilde{Y(\rho_v)}^i(g,X,t)v_i
\end{align*}
setting $\widetilde{Y(\rho_v)}(g,X,t) =\widetilde{Y(\rho_v)}^i(g,X,t)v_i$. On the other hand, the continuity of the coaction follows, in view of \cite[Lemma A.2.2]{BonnFlatGersPinc94}, from its linearity. Furthermore, it follows at once from the argument \eqref{aux-G-mod-HG-comod} that \eqref{aux-HG-coact} defines a left $\Fc^\infty(G)$-coaction.
\end{proof}

\subsection{Topological Lie-Hopf algebras}\label{subsect-Lie-Hopf-bicross}

In this subsection we shall revisit the matched pairs of Hopf algebras, from \cite{Majid-book,RangSutl,RangSutl-I-arxiv}, within the category of topological Hopf algebras. In view of these ideas, we shall construct the matched pair Hopf algebra $\Fc^\infty(G_2)\,\tacl\,U(\Fg_1)$, out of a matched pair $(G_1,G_2)$ of groups. We refer the reader to \cite{Majid-book,Maji90} for the details on the matched pairs of (Lie) groups, and matched pairs of Lie algebras. 

\medskip

\ni Let  $\Uc$ and $\Fc$ be two (topological) Hopf algebras. A right coaction
$$\Db:\Uc\lra\Uc\,\projot\, \Fc , \qquad \Db(u)  = u\ns{0}\, \projot\, u\ns{1}$$
equips $\Uc$ with a right $\Fc$-comodule coalgebra structure if the conditions
\begin{align*}
& u\ns{0}\ps{1}\,\projot\, u\ns{0}\ps{2}\,\projot\, u\ns{1}= u\ps{1}\ns{0}\,\projot\, u\ps{2}\ns{0}\,\projot\, u\ps{1}\ns{1}u\ps{2}\ns{1},\\
& \ve(u\ns{0})u\ns{1}=\ve(u)1,
\end{align*}
are satisfied for any $u\in \Uc$. One then forms a cocrossed product topological coalgebra $\Fc\,\topcl\,\Uc$, that has $\Fc\,\projot\, \Uc$ as the underlying t.v.s. and
\begin{align*}
&\Delta(f\,\topcl\, u)= f\ps{1}\,\topcl\, u\ps{1}\ns{0}\,\projot\,  f\ps{2}u\ps{1}\ns{1}\,\topcl\, u\ps{2}, \\
&\ve(f\,\topcl\, u)=\ve(f)\ve(u),
\end{align*}
as the topological coalgebra structure.

\medskip

\ni On the other hand, $\Fc$ is called a left $\Uc$-module algebra if $\Uc$ acts on $\Fc$
\begin{equation*}
\rt : \Uc\,\projot\, \Fc \lra \Fc 
\end{equation*}
such that
\begin{equation*}
u\rt 1=\ve(u)1, \qquad u\rt(fg)=(u\ps{1}\rt f)(u\ps{2}\rt g)
\end{equation*}
for any $u\in \Uc$, and any $f,g\in \Fc$. This time one endows the t.v.s. $\Fc\,\projot\, \Uc$ with an algebra structure, which is  denoted by $\Fc\,\topal\, \Uc$, with $1\,\topal\, 1$ as its unit and the multiplication given by
\begin{equation*}
(f\,\topal\, u)(g\,\topal\, v)=f (u\ps{1}\rt g)\,\topal\, u\ps{2}v.
\end{equation*}

\ni Finally, the pair $(\Fc, \Uc)$ of topological Hopf algebras is called a matched pair of Hopf algebras if $\Uc$ is a right $\Fc$-comodule coalgebra, $\Fc$ is a left $\Uc$-module algebra, and
\begin{align*}
&\ve(u\rt f)=\ve(u)\ve(f), \\ 
&\Delta(u\rt f)= u\ps{1}\ns{0} \rt f\ps{1}\,\projot\, u\ps{1}\ns{1}(u\ps{2}\rt f\ps{2}), \\
&\Db(1)=1\ot 1, \\ 
&\Db(uv)= u\ps{1}\ns{0} v\ns{0}\,\projot\, u\ps{1}\ns{1}(u\ps{2}\rt v\ns{1}),\\ 
& u\ps{2}\ns{0}\,\projot\, (u\ps{1}\rt f) u\ps{2}\ns{1}= u\ps{1}\ns{0}\,\projot\, u\ps{1}\ns{1}(u\ps{2}\rt f),
\end{align*}
for any $u\in\Uc$, and any $f\in \Fc$. One then forms the bicrossed product  Hopf algebra $\Fc\,\tacl\, \Uc$. It has $\Fc\,\topcl\, \Uc$ as the underlying coalgebra, $\Fc\,\topal\, \Uc$ as the underlying algebra, and its antipode is defined by
\begin{equation*}
S(f\,\tacl\, u)=(1\,\tacl\, S(u\ns{0}))(S(fu\ns{1})\,\tacl\, 1) , \qquad \forall\, f \in \Fc , \, \forall\,u \in \Uc.
\end{equation*}


\begin{definition}\label{Lie-module}
A left topological $\Fg$-module over a topological Lie algebra $\Fg$ is a t.v.s. $V$ such that the left $\Fg$-module structure map $\Fg\,\projot\, V \lra V$ is continuous.
\end{definition}

\begin{example}{\rm
Let $\Fg$ be a topological Lie algebra. The symmetric algebra $S(\Fg^\ast)$ is a (right) topological $\Fg$-module by the coadjoint action via \cite[Coroll. A.2.8]{BonnFlatGersPinc94}.
}\end{example}


\ni Now let $\Fc$ be a commutative topological Hopf algebra on which a topological Lie algebra $\Fg$ acts by derivations. We endow the vector space $\Fg\,\projot\,\Fc$ with the following bracket:
\begin{equation}\label{bracket}
[X\,\projot\, f, Y\,\projot\, g]= [X,Y]\,\projot\, fg+ Y\,\projot\, \ve(f)X\rt g- X\,\projot\, \ve(g) Y\rt f.
\end{equation}

\begin{lemma}
Let a topological Lie algebra $\Fg$ act on a commutative topological Hopf algebra $\Fc$ by derivations, and $\ve(X\rt f)=0$ for any $X\in \Fg$ and $f\in \Fc$. Then  the bracket \eqref{bracket} endows  $\Fg\,\projot\, \Fc$ with a topological Lie algebra strucure.
\end{lemma}

\begin{proof}
It is checked in \cite{RangSutl-I-arxiv}, see also \cite{RangSutl}, that the bracket is anti-symmetric, and that the Jacobi identity is satisfied. Moreover, since the bracket on $\Fg$, and the action of $\Fg$ on $\Fc$ are continuous, it follows that the bracket \eqref{bracket} is also continuous.
\end{proof}

\ni Next, let $\Fc$ coacts on $\Fg$ via $\Db_\Fg:\Fg\ra \Fg\,\projot\, \Fc$. Using the action of $\Fg$ on $\Fc$, and the coaction of  $\Fc$  on $\Fg$, we define an action of $\Fg$ on $\Fc\,\projot\, \Fc$ by
\begin{equation}\label{aux-bullet-action}
X\bullet (f^1 \,\projot\, f^2)= X\ns{0}\rt f^1\,\projot\, X\ns{1} f^2 + f^1\,\projot\, X\rt f^2.
\end{equation}
We note that since the $\Fc$-coaction on $\Fg$, and the $\Fg$ action on $\Fc$ are continuous, the action \eqref{aux-bullet-action} is also continuous.

\begin{definition}\label{def-Lie-Hopf}
Let a topological Lie algebra $\Fg$ act on a commutative topological Hopf algebra $\Fc$ by continuous derivations, and $\Fc$ coacts on $\Fg$ continuously. We say that $\Fc$ is a topological $\Fg$-Hopf algebra if
 \begin{enumerate}
   \item  the coaction $\Db_\Fg:\Fg\ra \Fg\,\projot\, \Fc$ is a map of topological Lie algebras,
   \item $\D$ and $\ve$ are $\Fg$-linear, i.e,  $\D(X\rt f)=X\bullet\D(f)$, \quad $\ve(X\rt f)=0$, \quad for any $f\in \Fc$, and any $X\in \Fg$.
    \end{enumerate}
\end{definition}

\ni Following \cite{RangSutl-I-arxiv} we extend the $\Fc$-coaction from $\Fg$ to $U(\Fg)$.

\begin{lemma}\label{U-coaction}
The extension of the coaction $\Db_\Fg:\Fg\ra \Fg\projot \Fc$  to $ \Db:U(\Fg)\ra U(\Fg)\projot \Fc,
$ via
\begin{align}\label{aux-coact-to-Ug}
\begin{split}
& \Db(uu')=u\ps{1}\ns{0}u'\ns{0}\,\projot\,u\ps{1}\ns{1}(u\ps{2}\rt u'\ns{1}),\qquad \forall\, u,u'\in U(\Fg),\\
& \Db(1)=1\,\projot\, 1,
\end{split}
\end{align}
is well-defined.
\end{lemma}

\begin{proof}
It is checked in \cite{RangSutl-I-arxiv} that \eqref{aux-coact-to-Ug} is well-defined. Hence, we just need to show that it is continuous, which follows from the linearity \cite[Lemma A.2.2]{BonnFlatGersPinc94}.
\end{proof}

\ni As a result, we obtain a topological version of \cite[Thm. 2.6]{RangSutl-I-arxiv} as follows.

\begin{theorem}\label{Theorem-Lie-Hopf-matched-pair}
Let $\Fc$ be a commutative topological $\Fg$-Hopf algebra. Then via the coaction of $\Fc$ on $U(\Fg)$ defined above and the natural action of $U(\Fg)$ on $\Fc$, the pair $(U(\Fg), \Fc)$ becomes a matched pair of topological Hopf algebras. Conversely, for a commutative topological Hopf algebra $\Fc$, if $(U(\Fg), \Fc)$ is a matched pair of topological Hopf algebras then $\Fc$ is a topological $\Fg$-Hopf algebra .
\end{theorem}

\ni We are now ready to construct our main example.

\begin{proposition}\label{prop-top-F-U}
Let $(G_1,G_2)$ be a matched pair of real analytic groups, with mutual analytical actions, and let $(\Fg_1,\Fg_2)$ be their Lie algebras. Then, $\Fc^\infty(G_2)$ is a $\Fg_1$-Hopf algebra.
\end{proposition}

\begin{proof}
Since the infinitesimal action
\begin{equation*}
\vp_X:G_2 \lra \Fg_1,\qquad \vp_X(y):=\left.\frac{d}{ds}\right|_{s=0}y\rt\exp(sX)
\end{equation*}
of $G_2$ on $\Fg_1$ is differentiable\footnote[1]{Follows from
\begin{equation*}
\exp(tY)\rt X = X + t\mu(X,Y,t), \qquad \mu(X,Y,0) = Y\rt X,
\end{equation*}
for any $X\in \Fg_1$, and any $Y\in \Fg_2$, which, in turn, follows from the action of $G_2$ on $G_1$ being analytical.}, it follows from Proposition \ref{prop-diff-ble-G-module} that $\Fg_1$ is a right $\Fc^\infty(G_2)$-comodule. Moreover, $\Fg_1$ acts on $\Fc^\infty(G_2)$ by 
\begin{equation*}
X(f):=\left.\frac{d}{ds}\right|_{s=0}\exp(sX)\rt f,
\end{equation*}
for any $X\in \Fg_1$, and any $f\in \Fc^\infty(G_2)$. Indeed,
\begin{align*}
& X(f)(y\exp(tY)) = \left.\frac{d}{ds}\right|_{s=0}(\exp(sX)\rt f )(y\exp(tY)) = \\
& \left.\frac{d}{ds}\right|_{s=0}f((y\exp(tY))\lt \exp(sX)) = \\
& \left.\frac{d}{ds}\right|_{s=0}f((y\lt (\exp(tY)\rt \exp(sX))) (\exp(tY)\lt \exp(sX))) = \\
& \left.\frac{d}{ds}\right|_{s=0}f((y\lt (\exp(tY)\rt \exp(sX)))\exp(tY)) + \\
&\hspace{4cm} \left.\frac{d}{ds}\right|_{s=0}f(y(\exp(tY)\lt \exp(sX)))
\end{align*}
for any $y\in G_2$, and any $Y\in \Fg_2$. Now, since $f\in \Fc^\infty(G_2)$,
\begin{align*}
& \left.\frac{d}{ds}\right|_{s=0}f((y\lt (\exp(tY)\rt \exp(sX)))\exp(tY)) = \\
& \left.\frac{d}{ds}\right|_{s=0}f(y\lt (\exp(tY)\rt \exp(sX))) + \left.\frac{d}{ds}\right|_{s=0} t\widetilde{f}((y\lt (\exp(tY)\rt \exp(sX))),Y,t)
\end{align*}
for some continuous $\widetilde{f}:G_2 \times \Fg_2 \times \Rb \to \Rb$, where
\begin{align*}
& \left.\frac{d}{ds}\right|_{s=0}f(y\lt (\exp(tY)\rt \exp(sX))) = (\exp(tY)\rt X) (f)(y) = \\
& X( f)(y) + t\psi_X(e, Y, t)( f)(y)
\end{align*}
for some continuous $\psi_X: G_2 \times \Fg_2 \times \Fg_1 \to \Fg_1$, due to the fact that the (left) action of $G_2$ on $\Fg_1$ is differentiable.We note also that the composition $(y,Y,t) \mapsto \psi_X(e, Y, t)( f)(y)$ yields yet another continuous map $G_2 \times \Fg_2 \times \Rb \to \Rb$. Finally,
\[
f(y(\exp(tY)\lt \exp(sX))) = f(y) + tf'(y, Y\lt \exp(sX), t)
\]
for some continuous $f': G_2\times \Fg_2 \times \Rb \to \Rb$, and hence 
\[
\left.\frac{d}{ds}\right|_{s=0}f(y(\exp(tY)\lt \exp(sX))) = tf'(y, Y\lt X, t).
\]
As a result, we see that
\[
X(f)(y\exp(tY)) = X( f)(y) + t\widetilde{X(f)}(y,Y,t),
\]
where
\[
\widetilde{X(f)}:G_2\times \Fg_2 \times \Rb \to \Rb, \qquad \widetilde{X(f)}(y,Y,t) := \psi_X(e, Y, t)( f)(y) + f'(y, Y\lt X, t).
\]
Next, for any $Z \in \Fg_2$, let
\[
Z(X(f)):G_2 \lra \Rb, \qquad Z(X(f))(y) := \widetilde{X(f)}(y,Z,0) = (Z\rt X)(f)(y) + (Z\lt X)(f)(y).
\]
Then the claim follows from the observation that both $(Z\rt X)(f) , (Z\lt X)(f)\in \Fc^\infty(G_2)$; the former by the very arguments above, and the latter by $\Fc^\infty(G_2)$ being a differentiable $G_2$-module, \cite[Sect. 4]{HochMost62}.
\end{proof}

\ni Hence, by Theorem \ref{Theorem-Lie-Hopf-matched-pair} and Proposition \ref{prop-top-F-U}, we have the Hopf algebra $\Fc^\infty(G_2) \tacl U(\Fg_1)$.

\section{Hopf-cyclic cohomology for topological Hopf algebras}\label{sect-Hopf-cyclic-top}

In this section we first  recall  the basic definitions and results for Hopf-cyclic cohomology with coefficients in the category of topological Hopf algebras. We continue by  characterizing  the category of coefficient spaces  (SAYD modules) over a topological Hopf algebra $\Hc$ as the category of representations of a topological algebra associated to the Hopf algebra $\Hc$.

\subsection{Hopf-cyclic complex for topological Hopf algebras}\label{subsect-Hopf-cyclic-comp}

\ni We shall include, in this subsection, a brief discussion of Hopf-cyclic cohomology with coefficients in the category of topological Hopf algebras. To this end, we adopt the categorical viewpoint of \cite{Conn83}, see also \cite{Loday-book}, to consider cocyclic modules in the category of topological Hopf algebras.

\ni Let $\Hc$ be a topological Hopf algebra. A character $\d: \Hc\lra k$ is a continuous unital algebra map, and a group-like element $\s\in \Hc$ is the dual object of the character, \ie $\D(\s)=\s\,\projot\, \s$ and $\ve(\s) = 1$. The pair $(\d,\s)$ is called a modular pair in involution (MPI) if
\begin{equation*}
\d(\s)=1, \quad \text{and}\quad  S_\d^2=\Ad_\s,
\end{equation*}
where $\Ad_\s(h)= \s h\s^{-1}$, and $S_\d(h)=\d(h\ps{1})S(h\ps{2})$.

\ni In the presence of a topology, we shall extend the scope of the canonical MPI associated to $\Fc\acl U(\Fg)$, see \cite[Thm. 3.2]{RangSutl}, to $\Fc\tacl U(\Fg)$, at which case the Lie algebra $\Fg$ is allowed to be infinite dimensional.

\ni It follows from \cite[Prop. 1(iii)]{BonnSter05} that $\Fg^\circ\,\projot\,\Fg \cong \End(\Fg)$, and hence there exists a well-defined element
\begin{equation*}
\rho = \sum_{i\in I} f^i\,\projot\,X_i \in \Fg^\circ\,\projot\,\Fg,
\end{equation*}
that corresponds to $\Id \in \End(\Fg)$. This, in turn, results in a well-defined functional
\begin{equation}\label{aux-delta-tr-ad}
\d_\Fg := \sum_{i\in I} X_i\rt f^i : \Fg\lra k.
\end{equation}
Note that in case $\Fg$ is finite dimensional, $\{X_i\,|\,i\in I\}$ and $\{f^i\,|\,i\in I\}$ are dual pair of bases, and we have $\d_\Fg = \Tr\circ \ad$.

\ni Next, let $\Fg=\cup_{n\in \Nb}V_n$ be a sequence of definition \cite[Sect. 13]{Treves-book}, compatible with the $\Fc$-coaction, that is, the right $\Fc$-coaction $\nb:\Fg\to \Fg\,\projot\,\Fc$ restricts to $\nb:V_n\to V_n\,\projot\,\Fc$, for any $n\in \Nb$. 

\ni For instance, if $H$ is a Lie group, and $\Fg$ a Lie algebra admitting a locally finite (differentiable) left $H$-module structure, then $\Fg = \sum_{n\in \Nb}V_n$ for the finite dimensional $H$-submodules $V_n$, $n\in \Nb$. In this case, $\nb:\Fg\to \Fg\,\projot\,\Fc^\infty(H)$ restricts to $\nb:V_n\to V_n\,\projot\,\Fc^\infty(H)$ for any $n\in \Nb$. This, in turn, is quite likely to happen in many examples. If $(G_1,G_2)$ is a matched pair of Lie groups, with Lie algebras $(\Fg_1,\Fg_2)$, then $\Fg_1$ being a finite dimensional (left) $G_2$-module, it is a locally finite $G_2$-module. Furthermore, in the infinite dimensional case there is the ``primitive Lie-Cartan pseudogroups'' of \cite{Cart09}, whose matched pair decompositions are discussed in \cite{MoscRang09}. More explicitly, for any primitive Lie-Cartan pseudogroup $\Pi$, letting $\Diff_\Pi$ be the globally defined diffeomorphisms of type $\Pi$, the Kac decomposition \cite{Kac68} yields $\Diff_\Pi = G_\Pi \times N_\Pi$, where $G_\Pi$ is the subgroup of all affine transformations of $\Rb^n$ which are in $\Diff_\Pi$, and $N_\Pi$ is the subgroup of those diffeomorphisms of $\Rb^n$ that preserve the origin to order 1. Therefore, we see that even in this case, we have a finite dimensional  Lie algebra $\Fg_\Pi$ of $G_\Pi$, being a locally finite $N_\Pi$-module.

\ni For any such topological $\Fg$-Hopf algebra $\Fc$, let
\begin{equation}\label{aux-sigma-det}
\s = \underset{n\lra\infty}{\lim}\s_n \in \Fc,
\end{equation}
where $\s_n \in \Fc$ is defined as the determinant of the first order matrix coefficients on the finite dimensional $\cup_{k=1}^nV_k \subseteq \Fg$. By \cite[Lemma 3.1]{RangSutl}, $\s_n$ is a group-like for any $n\in \Nb$, and since the comultiplication $\D:\Fc\lra \Fc\,\projot\,\Fc$ is continuous, we have $\D(\s) = \D(\underset{n\lra\infty}{\lim}\s_n)=\underset{n\lra\infty}{\lim}\D(\s_n) = \s\,\projot\,\s$. That is, $\s\in \Fc$ is also a group-like.

\begin{lemma}
The group-like $\s\in \Fc$ is independent of the choice of the sequence of definition.
\end{lemma}

\begin{proof}
Let $\Fg=\cup_{n\in\Nb}V_n = \cup_{n\in\Nb}W_n$ be two sequence of definitions. Accordingly we have two sequences $(\s_n)_{n\in\Nb} \subseteq \Fc$, obtained as the first order matrix coefficients of the coaction $V_n\lra V_n\,\projot\,\Fc$, and $(\s'_n)_{n\in\Nb} \subseteq \Fc$, obtained as the first order matrix coefficients of the coaction $W_n\lra W_n\,\projot\,\Fc$. Let then $\s := \underset{n\lra\infty}{\lim}\s_n \in \Fc$, and similarly $\s' = \underset{n\lra\infty}{\lim}\s'_n\in\Fc$.

\ni Since a sequence of definition consists of an increasing sequence of subspaces, for any $n\in \Nb$, there is $n_0\in\Nb$ such that $V_n \subseteq W_{n_0}$, and hence $\s'\in\Fc$ agrees with $\s\in\Fc$ on $V_n$ for any $n\in\Nb$. We thus conclude that $\s=\s'$.
\end{proof}

\ni Accordingly, we have the following generalization of \cite[Thm. 3.2]{RangSutl}.

\begin{theorem}
Let  $\Fg$ be a topological Lie algebra, and $\Fc$ a topological $\Fg$-Hopf algebra. Let also $\Fg$ have a sequence of definition compatible with the $\Fc$-coaction. Then $\d:U(\Fg)\lra k$ being the extension of the functional $\d_\Fg:\Fg\lra k$ of \eqref{aux-delta-tr-ad}, and $\s\in \Fc$ the group-like defined by \eqref{aux-sigma-det}, the pair $(\d,\s) := (\ve\,\projot\,\d,\,\s\,\projot\,1 )$ is an MPI for the Hopf algebra $\Fc\,\tacl\, U(\Fg)$.
\end{theorem}

\begin{proof}
It suffices to prove the MPI condition on $f\,\tacl\,1,\, 1 \,\tacl\,X \in \Fc\,\tacl\, U(\Fg)$ for any $f \in \Fc$, and any $X \in \Fg$. Since a sequence of definition consists of an increasing sequence of subspaces, for any $X \in \Fg$ there is $n_0 \in \Nb$ such that $X \in \cup_{n=1}^{n_0}V_n$, which is finite dimensional. The claim then follows essentially from the proof of \cite[Thm. 3.2]{RangSutl}, see also \cite[Thm. 3.2]{RangSutl-I-arxiv}.
\end{proof}


\ni Stable anti-Yetter-Drinfeld (SAYD) modules appeared  first in \cite{HajaKhalRangSomm04-I,JaraStef06} as the generalizations of modular pairs in involution \cite{ConnMosc98} . In the rest of this subsection we upgrade them to the level of topological Hopf algebras.

\begin{definition}
Let $V$ be a topological right $\Hc$-module by $V \,\projot\, \Hc \to V$, $v\,\projot\,h \to v\cdot h$, and left $\Hc$-comodule via $\Db:V \to \Hc\,\projot\, V$, $\Db(v)=v\ns{-1}\,\projot\,v\ns{0}$. We say that $V$ is an AYD module over $\Hc$ if
\begin{equation*}
\Db(v\cdot h)= S(h\ps{3})v\ns{-1}h\ps{1}\,\projot\, v\ns{0}\cdot h\ps{2},
\end{equation*}
for any $v\in V$ and $h\in \Hc$. Moreover, $V$ is called stable if
\begin{equation*}
v\ns{0} \cdot v\ns{-1}=v,
\end{equation*}
for any $v\in V$.
\end{definition}

\ni Similar to the algebraic case, any MPI defines a one dimensional SAYD module and all one dimensional SAYD modules come this way.

\begin{proposition}
Let $\Hc$ be a topological Hopf algebra, $\s\in\Hc$ a group-like element, and $\d\in\Hc^\circ$ a character. Then, $(\d,\s)$ is an MPI if and only if $\,^\s k_\d$ is a SAYD module over $\Hc$.
\end{proposition}

\ni We next define the Hopf-cyclic cohomology, with SAYD coefficients, of topological module coalgebras. To this end we first recall the tensor product of topological modules over a topological algebra from \cite[Def. 1.7]{Tayl72}. Let $A$ be a topological algebra, $M$ a topological right $A$-module, and $N$ a topological left $A$-module. Then, $M\,\projhat_A\,N$ is defined to be the vector space $(M\times N)/W$, equipped with the quotient topology, where
\begin{equation*}
W:={\rm Span}\{(ma,n)-(m,an)\mid m\in M,\,n\in N,\,a\in A\} \subseteq M\times N.
\end{equation*}
We note also from \cite[Prop. 4.5]{Treves-book} that, as a topological space, $M\,\projhat_A\,N$ is Hausdorff if and only if $W$ is closed, see also \cite[Prop. 1.5]{Tayl72}, as well as \cite[Sect. II.6.1]{Schaefer-book} to note that the quotient topology is also locally convex.

\ni Let $V$ be a right-left SAYD module over a topological Hopf algebra $\Hc$, and $\Cc$ a topological $\Hc$-module coalgebra via the continuous action $\Hc\,\projot\, \Cc \lra \Cc$, that is,
\begin{equation*}
\D(h\cdot c) = (h\ps{1} \cdot c\ps{1})\,\projot\,(h\ps{2} \cdot c\ps{2}),\qquad \ve(h\cdot c) = \ve(h)\ve(c).
\end{equation*}
We then define the topological Hopf-cyclic complex of $\Cc$, with coefficients in $V$, under the symmetry of $\Hc$ by
\begin{equation*}
C_{\top}^n(\Cc,\Hc,V):= V\,\projhat_\Hc\, \Cc^{\projot\, n+1}, \quad n\ge 0,
\end{equation*}
equipped with the face operators
\begin{align*}
\begin{split}
& \p_i: C_{\top}^n(\Cc,\Hc,V)\lra C_{\top}^{n+1}(\Cc,\Hc,V), \qquad 0\le i\le n+1,\\
&\p_i(v\,\projhat_\Hc\, c^0\,\oprojdots\, c^n)=v\,\projhat_\Hc \,c^0\,\oprojdots\, c^i\ps{1}\,\projhat\, c^i\ps{2}\,\oprojdots\, c^n, \quad 0\le i \leq n,\\
&\p_{n+1}(v\,\projhat_\Hc \,c^0\,\oprojdots \,c^n)=v\ns{0}\,\projhat\, c^0\ps{2} \projhat \,c^1\,\oprojdots\, c^n\,\projhat \, v\ns{-1}\cdot c^0\ps{1},
\end{split}
\end{align*}
the degeneracy operators
\begin{align*}
\begin{split}
& \s_j: C_{\top}^n(\Cc,\Hc,V)\lra C_{\top}^{n-1}(\Cc,\Hc,V), \qquad 0\le j\le n-1, \\
&\s_j (v\,\projhat_\Hc \,c^0\,\oprojdots \,c^n)= v\,\projhat_\Hc \, c^0\,\oprojdots\, \ve(c^{j+1})\,\oprojdots\, c^n,
\end{split}
\end{align*}
and the cyclic operator
\begin{align}\label{aux-cyclic-opt-module-coalg}
\begin{split}
& \tau: C_{\top}^n(\Cc,\Hc,V)\lra C_{\top}^{n}(\Cc,\Hc,V), \\
&\tau(v\,\projhat_\Hc \,c^0\,\oprojdots \,c^n)=v\ns{0}\,\projhat_\Hc \, c^1\,\oprojdots\, v\ns{-1}\cdot c^0.
\end{split}
\end{align}
Then, the graded module $C^\ast_{\top}(\Cc,\Hc,V)$ becomes a cocyclic module \cite{HajaKhalRangSomm04-II}, with the Hochschild coboundary
\begin{equation*}
b: C_{\top}^{n}(\Cc,\Hc,V)\lra C_{\top}^{n+1}(\Cc,\Hc,V), \qquad  b:=\sum_{i=0}^{n+1}(-1)^i\p_i,
\end{equation*}
and the Connes boundary operator
\begin{equation*}
B: C_{\top}^{n}(\Cc,\Hc,V)\lra C_{\top}^{n-1}(\Cc,\Hc,V), \qquad B=\sum_{k=0}^n(-1)^{kn}\tau^k\s_{-1},
\end{equation*}
which satisfy $b^2=B^2=(b+B)^2=0$ by \cite{Conn83}. We note that
\begin{equation*}
\s_{-1} = \s_{n-1}\circ\tau:C_{\top}^{n}(\Cc,\Hc,V)\lra C_{\top}^{n-1}(\Cc,\Hc,V)
\end{equation*}
is the extra degeneracy operator.

\begin{definition}
The Hopf-cyclic cohomology of the topological $\Hc$-module coalgebra $\Cc$, with coefficients in a SAYD module $V$ over $\Hc$, is  denoted by $HC^\ast_{\top}(\Cc,\Hc,V)$, and is defined to be the total cohomology of the bicomplex
\begin{align*}
CC_{\top}^{p,q}(\Cc,\Hc,V) = \begin{cases} 
V\,\projhat_\Hc \Cc^{\projot\, q-p},& \text{if} \quad 0\le p\le q, \\
0, & \text{otherwise.}
\end{cases}
\end{align*}
The periodic Hopf-cyclic cohomology of the topological $\Hc$-module coalgebra $\Cc$, with coefficients in a SAYD module $V$ over $\Hc$, is denoted by $HP^\ast_{\top}(\Cc,\Hc,V)$, and it is defined to be the total cohomology of the bicomplex
\begin{align*}
CC_{\top}^{p,q}(\Cc,\Hc,V) = \begin{cases} 
V\,\projhat_\Hc \Cc^{\projot\, q-p},& \text{if} \quad p\le q, \\
0, & \text{otherwise.}
\end{cases}
\end{align*}
\end{definition}

\ni Furthermore, it follows from the proof of \cite[Prop. 1.5]{Tayl72} that, in case $\Cc=\Hc$ a module coalgebra with the left regular action of $\Hc$, the map
\begin{align}
\begin{split}
& \Ic:V\,\projhat_\Hc\,\Hc^{\projhat\,n+1} \lra V\,\projot\,\Hc^{\projot\,n} \\
& v\,\projhat_\Hc\,h^0\,\oprojdots\,h^n\mapsto v\cdot h^0\ps{1}\,\projot\,S(h^0\ps{2})(h^1\,\oprojdots\,h^n)
\end{split}
\end{align}
is a homeomorphism. In this case, the cocyclic structure is transformed into the one with the face operators
\begin{align*}
\begin{split}
& \p_i: C_{\top}^n(\Hc,V)\ra C_{\top}^{n+1}(\Hc,V), \qquad 0\le i\le n+1,\\
&\p_0(v\,\projot\, h^1\oprojdots h^n)=v\ot 1\ot h^1\oprojdots h^n,\\
&\p_i(v\ot h^1\oprojdots h^n)=v\ot h^1\oprojdots h^i\ps{1}\ot h^i\ps{2}\oprojdots h^n, \quad 1\le i \leq n,\\
&\p_{n+1}(v\ot h^1\oprojdots h^n)=v\ns{0}\ot h^1\oprojdots h^n\ot v\ns{-1},
\end{split}
\end{align*}
the degeneracy operators
\begin{align*}
\begin{split}
& \s_j: C_{\top}^n(\Hc,V)\ra C_{\top}^{n-1}(\Hc,V), \qquad 0\le j\le n-1, \\
&\s_j (v\ot h^1\oprojdots h^n)= v\ot h^1\oprojdots \ve(h^{j+1})\oprojdots h^n,
\end{split}
\end{align*}
and the cyclic operator
\begin{align*}
\begin{split}
& \tau: C_{\top}^n(\Hc,V)\ra C_{\top}^{n}(\Hc,V), \\
&\tau(v\ot h^1\oprojdots h^n)=v\ns{0}h^1\ps{1}\ot S(h^1\ps{2})\cdot(h^2\oprojdots h^n\ot v\ns{-1}).
\end{split}
\end{align*}

\subsection{Characterization of Hopf-cyclic coefficients}\label{subsect-Hopf-cyclic-coeff}

In this subsection we upgrade the content of \cite[Prop. 4.2]{HajaKhalRangSomm04-I}, which is limited to finite dimensional Hopf algebras, to topological Hopf algebras, which are now allowed to be (countably) infinite dimensional. More precisely, we shall identify the category of SAYD modules over a topological Hopf algebra $\Hc$, with the representation category of an algebra $B_{AYD}(\Hc)$ associated to $\Hc$.

\ni To this end, we recall from \cite[Prop. 4.2]{HajaKhalRangSomm04-I} that if $H$ is finite dimensional, then $V$ is a AYD module over $H$ if and only if it is a module over $B_{AYD}(H) := H^\ast \ot H$, the algebra structure of which is given by
\begin{equation}\label{B-AYD}
(\vp \ot h)\cdot (\vp' \ot h') = {\vp'}\ps{1}(S^{-1}(h\ps{3})){\vp'}\ps{3}(S^2(h\ps{1}))\vp{\vp'}\ps{2} \ot h\ps{2}h'.
\end{equation}

\ni Now let $\Hc$ be a topological Hopf algebra, $\Hc^{\circ}$ be its (strong) dual, and let $B_{AYD}^{\top}(\Hc):=\Hc^\circ\,\projot\,\Hc$ be the algebra whose multiplication is the extension of \eqref{B-AYD} to $\projot$. 

\ni It follows from \cite[Prop. 1(iii)]{BonnSter05}, see also \cite{BonnFlatGersPinc94}, that there is an isomorphism
\begin{equation*}
\lambda:B_{AYD}^{\top}(\Hc) \lra \End(\Hc), \qquad \lambda(\vp\,\projot\,h)(x)=\vp(x)h
\end{equation*}
of vector spaces. As a result, we have an element
\begin{equation}\label{aux-rho}
\rho := \sum_{i\in I} f^i\,\projot\,x_i \in B_{AYD}^{\top}(\Hc)
\end{equation}
corresponding to the identity $\Id_\Hc\in \End(\Hc)$, \ie $\lambda(\rho)=\Id_\Hc$. We thus record below the following generalization of \cite[Prop. 4.2]{HajaKhalRangSomm04-I} as the main result of the present subsection.

\begin{proposition}\label{prop-classify-AYD}
Let $\Hc$ be a topological Hopf algebra. If $V$ is a right-left AYD module over $\Hc$, then it is a right $B_{AYD}^{\top}(\Hc)$-module via
\begin{equation*}
v \cdot (\vp\,\projot\,h) := \vp(v\ns{-1})v\ns{0}\cdot h.
\end{equation*}
Conversely, if $\Hc$ is a topological Hopf algebra and  $V$ is a right $B_{AYD}^{\top}(\Hc)$-module, then it is a right-left AYD module over $\Hc$ by the right $\Hc$-action
\begin{equation*}
v\cdot h := v \cdot (\ve\,\projot\,h),
\end{equation*}
and the left $\Hc$-coaction
\begin{equation*}
\nb_\Hc^L(v):=\sum_{i\in I}x_i\,\projot\,v\cdot(f^i\,\projot\,1).
\end{equation*}
\end{proposition}

\ni As for the stability we have the following result.

\begin{proposition}
Let $\Hc$ be a topological Hopf algebra, and $V$ a right $B_{AYD}^{\top}(\Hc)$-module. Then $V$ is stable if and only if the fixed point set of $\rho$ is $V$.
\end{proposition}

\begin{example}{\rm
Let  $H$ be a countable dimensional Hopf algebra, equipped with the natural topology. Then,
\begin{enumerate}
\item an AYD module over $H$,
\item $H^\ast\ot H$,
\end{enumerate}
are right modules over $B_{AYD}(H)=B_{AYD}^{\top}(H)$, the former by Proposition \ref{prop-classify-AYD}, and the latter by the right regular action. Hence they make countable dimensional coefficient spaces for the topological Hopf-cyclic cohomology, under the symmetry of $H$.
}\end{example}

\begin{remark}
{\rm
If $\Hc$ is a topological Hopf algebra, $V$ is a right / left SAYD module over $\Hc$, and $\Cc$ is a topological left $\Hc$-module coalgebra, then the cyclic operator \eqref{aux-cyclic-opt-module-coalg} can be given by the element $\rho\in B_{AYD}^{\top}(\Hc)$ of \eqref{aux-rho} as
\begin{equation*}
\tau(v\,\projhat_\Hc \,c^0\,\oprojdots \,c^n)=\sum_{i\in I}(v\cdot f^i)\,\projhat_\Hc \, c^1\,\oprojdots\, x_i\cdot c^0
\end{equation*}
for any $v\in V$, and any $c^0,\ldots,c^n\in \Cc$.
}
\end{remark}

\begin{remark}{\rm
If  $V$ is an AYD module over a Hopf algebra $\Hc$, then its (strong) dual $V^{\circ}$ is naturally an AYD-contra-module \cite{Brze11}, which is, in view of  Proposition \ref{prop-classify-AYD}, a left module over $B^{\top}_{AYD}$.  This leads us to a classification of AYD contra-modules, and has a direct application  in cup products in Hopf-cyclic cohomology. This is the  ground for a detailed discussion  in an upcoming paper \cite{RangSutl-VI}.
}\end{remark}

\begin{remark}{\rm
We finally note that the topological Hopf-cyclic cohomology is a natural generalization of the algebraic one. Any (countable dimensional) algebraic Hopf algebra $H$, and any (countable dimensional) algebraic SAYD module  
$V$ over $H$ can be topologized with the strict inductive limit topology \cite{Treves-book} (the natural topology \cite[Appdx. 2]{BonnFlatGersPinc94}). In this case, the algebraic and the topological Hopf-cyclic complexes coincide, \cite[Prop. A.2.7(ii)]{BonnFlatGersPinc94}.
}\end{remark}

\section{Cyclic cohomology for topological Lie algebras}\label{sect-cycl-cohom-top-Lie}

\ni In this section we review the cyclic cohomology theory for Lie algebras in the presence of a suitable topology. The advantage of the topological point of view manifests itself on the corepresentation categories of a topological Lie algebra $\Fg$, and its universal enveloping algebra $U(\Fg)$ viewed as a topological Hopf algebra. More precisely, in Subsection \ref{subsect-comod-lie-alg} we remove the conilpotency condition of \cite[Def. 5.4]{RangSutl-II}, as such, we may consider the Weil complex $S(\Fg^\ast) \ot \wg^\ast \Fg^\ast$, as a Lie algebra cyclic cohomology complex; compare with \cite[Ex. 5.6]{RangSutl-II}. In Subsection \ref{subsect-cyclic-complexes-lie-alg}, on the other hand, we discuss two cyclic cohomology theories associated to a topological Lie algebra.

\subsection{Corepresentations of topological Lie algebras}\label{subsect-comod-lie-alg}

\ni In this subsection we shall study the categories of comodules over a topological Lie algebra $\Fg$, and comodules over its universal enveloping Hopf algebra $U(\Fg)$.

\begin{definition}
Let $\Fg$ be a topological Lie algebra, and $V$ a t.v.s. Then $V$ is called a topological left $\Fg$-comodule if there exists a linear map
\begin{equation*}
\nb_\Fg:V \lra \Fg\, \widehat{\ot}_\pi V,\quad \nb_\Fg(v):=v\nsb{-1}\,\projot\, v\nsb{0}
\end{equation*}
such that
\begin{equation*}
v\nsb{-2}\wg v\nsb{-1}\,\projot\, v\nsb{0} = 0.
\end{equation*}
\end{definition}

\begin{example}{\rm
The symmetric algebra $S(\Fg^\ast)$, equipped with the natural topology, is a (left) topological $\Fg$-comodule by the Koszul coaction
\begin{align}\label{Koszul-coact}
\begin{split}
& \nb^\ast_\Fg:S(\Fg^\ast)\lra \Fg\, \projot\, S(\Fg^\ast),\\
& \nb^\ast_\Fg(R)(Y_0,\ldots,Y_q) = \sum_{k=0}^{q} Y_kR(Y_0,\ldots,\widehat{Y}_k,\ldots,Y_q),
\end{split}
\end{align}
for any $Y_0,\ldots, Y_q \in \Fg$, and any $R\in S_q(\Fg^\ast)$, \cite{Nats78}. 
}\end{example}


\ni Let us recall  from \cite[Prop. 5.2]{RangSutl-II} that  any  $\Fg$-comodule $V$ is naturally a module over $S(\Fg^\ast)$ via
\begin{equation}\label{aux-Sg-ast-action}
v\cdot f=f(v\nsb{-1})v\nsb{0}.
\end{equation}

\ni This correspondence is revisable in the topological case as well.

\begin{proposition}\label{prop-g-comod-Sg-mod}
Let $\Fg$ be a topological Lie algebra, and $V$ a t.v.s. Then, $V$ is a topological left $\Fg$-comodule if and only if it is a topological right $S(\Fg^\ast)$-module via \eqref{aux-Sg-ast-action}.
\end{proposition}

\begin{proof}
Let $V$ be a topological left $\Fg$-comodule via $\nb:V\lra \Fg\,\projot\, V$, $\nb(v)=v\nsb{-1}\,\projot\, v\nsb{0}$. Then since the $\Fg$-coaction is continuous by the assumption, the (right) $S(\Fg^\ast)$-action given by \eqref{aux-Sg-ast-action} is also continuous.

\ni Conversely let $V$ be a topological right $S(\Fg^\ast)$-module via an action $V\projot S(\Fg^\ast)\lra V$, $v\,\projot\,f\mapsto v\cdot f$. Using the element $\sum_{i\in I}f^i\,\projot\,X_i \in \Fg^\circ\,\projot\,\Fg \cong \End(\Fg)$ that corresponds to $\Id\in \End(\Fg)$, we have $v\cdot f = \sum_{i\in I}f(X_i)v\cdot f^i$. Hence the claim follows from the fact that $v\mapsto X_i\,\projot\,v\cdot f^i$ defines a (left) $\Fg$-coaction.
\end{proof}

\begin{corollary}
Given a topological Lie algebra $\Fg$, the category of topological left $\Fg$-comodules, and the category of topological right $S(\Fg^\ast)$-modules are isomorphic.
\end{corollary}

\begin{proof}
It follows at once from the above Proposition \ref{prop-g-comod-Sg-mod} that the isomorphism is given, on the level of objects, by the identity functor. Let us observe that this is also the case on the level of morphisms. Indeed, is $\vp:V \to W$ is a morphism of (left) $\Fg$-comodules, \ie
\[
\xymatrix{
V\ar[rr]^\vp \ar[d]_{\nb_V} & & W\ar[d]^{\nb_W} \\
\Fg \,\projot\, V \ar[rr]_{\Id_\Fg \,\projot\, \vp} & & \Fg \,\projot\, W
}
\]
is commutative, then for any $v \in V$, and any $f \in S(\Fg^\ast)$,
\[
\vp(v) \cdot f = f(\vp(v)\nsb{-1})\vp(v)\nsb{0} = f(v\nsb{-1})\vp(v\nsb{0}) = \vp(v\cdot f),
\]
that is,
\[
\xymatrix{
V \ar[rr]^\vp && W \\
V \,\projot\, S(\Fg^\ast)  \ar[rr] _{\vp\,\projot\, \Id_{S(\Fg^\ast)}}\ar[u]_{\rho_V}&& W \,\projot\, S(\Fg^\ast)\ar[u]_{\rho_V}
}
\]
is also commutative, in other words, $\vp:V\to W$ is a morphism of (right) $S(\Fg^\ast)$-modules.
\end{proof}

\ni Next, a topological analogue of \cite[Lemma 5.3]{RangSutl-II} is in order.

\begin{proposition}\label{prop-Ug-coact-to-g-coact}
Let $V$ be a t.v.s., $\Fg$ a topological Lie algebra, and $U(\Fg)$ be its universal enveloping Hopf algebra. If $V$ is a topological left $U(\Fg)$-comodule, then it is a topological left $\Fg$-comodule.
\end{proposition}

\begin{proof}
It follows from \cite[Lemma A.2.2]{BonnFlatGersPinc94} that the canonical projection $\pi:U(\Fg) \lra \Fg$, and similarly $\pi \ot \Id:U(\Fg)\,\projot\, V \lra \Fg\, \projot\, V$ is continuous. Thus the claim follows.
\end{proof}





\ni The next proposition is about the reverse direction which is the one that algebraic comodules fail to take. However, we shall need the following terminology.





\begin{definition}\label{def-fixed}
A t.v.s $V$ is called fixed-bounded if for any fixed $v\in V$ there is a family of seminorms $\{q_\lb\mid\;\lb\in\Lb\}$, that defines the topology of $V$, such that
\begin{equation*}
\sup_{\lb\in \Lb}q_\lb (v)<\infty.
\end{equation*}
\end{definition}

\begin{example}{
\rm Any normable space is fixed-bounded.}
\end{example}

\ni We are now ready to prove our main result in this section.


\begin{proposition}\label{prop-g-coact-to-Ug-coact}
Let $\Fg$ be a topological Lie algebra, $U(\Fg)$ its universal enveloping Hopf algebra, and $V$ be a fixed-bounded left $\Fg$-comodule by $\nb_\Fg:V \lra \Fg \widehat{\ot}_\pi V$, where $\nb_\Fg(v):=v\nsb{-1}\,\projot\, v\nsb{0}$. Then, $V$ is a topological $U(\Fg)$-comodule via
\begin{equation}\label{aux-g-to-U(g)-comod}
\nb_{\Fg}^{\exp}:V \lra U(\Fg)\, \projot\, V,\quad v \mapsto \sum_{k=0}^\infty\frac{1}{k!}v\nsb{-k}\ldots v\nsb{-1} \,\projot\, v\nsb{0}.
\end{equation}
\end{proposition}

\begin{proof}
Let $v\in V$, and $\{q_\lb\mid \lb\in \Lb\}$ is the family given in Definition \ref{def-fixed}.  We first  show that the series \eqref{aux-g-to-U(g)-comod} converges. To this end, we consider the sequence
\begin{equation}\label{aux-sequence-part-sums}
(s_n)_{n \in \Nb},\qquad s_n := \sum_{k=0}^n\frac{1}{k!}v\nsb{-k}\ldots v\nsb{-1} \,\projot\, v\nsb{0}
\end{equation}
of partial sums. For any continuous semi-norm $\rho$ 
on $U(\Fg)\,\projot\, V$, 
using \cite[Prop. 7.7]{Treves-book} and the continuity of the $\Fg$-coaction, for any $n\geq m$ we have $\lb_i\in\Lb$ such that
\begin{align*}
\begin{split}
& \rho(s_n-s_m) \leq \sum_{k=m+1}^n\frac{1}{k!}\rho(v\nsb{-k}\ldots v\nsb{-1}\projot  v\nsb{0}) \\
&\leq \sum_{k=m+1}^n\frac{q_{\lb_k}(v)}{k!} \le  \sup_{\lb\in \Lb}q_{\lb} (v)  \sum_{k=m+1}^n\frac{ 1}{k!}.
\end{split}
\end{align*}
 Therefore, the sequence \eqref{aux-sequence-part-sums} of partial sums is Cauchy, and hence converges in $U(\Fg)\,\projot\, V$. Finally, by \cite[Prop. 5.7]{RangSutl-II}, \eqref{aux-g-to-U(g)-comod} is indeed a coaction.
\end{proof}

\ni The following simple example illustrates that the above result, Proposition \ref{prop-g-coact-to-Ug-coact}, strictly improves \cite[Prop. 5.7]{RangSutl-II}. It simply shows that in the presence of a suitable topology, a $\Fg$-coaction gives rise to a $U(\Fg)$-coaction without the (local) conilpotency condition of \cite[Prop. 5.7]{RangSutl-II}.

\begin{example}
{\rm
Let $\Fg=\langle X \rangle$ be a one dimensional trivial Lie algebra. Then $k\lra \Fg\,\projot\,k$, $\one\mapsto X\,\projot\,\one$ determines a left $\Fg$-coaction on $k$. As a result of Proposition \ref{prop-g-coact-to-Ug-coact} we have the exponentiation
\begin{equation*}
k\lra U(\Fg)\,\projot\,k,\qquad \one\mapsto \exp(X)\,\projot\,\one.
\end{equation*}
}
\end{example}




\ni On the next step, we shall prove an analogue of Proposition \ref{prop-g-coact-to-Ug-coact} for $S(\Fg^\ast)$, equipped with the Koszul coaction. However, the proof will require special attention as $S(\Fg^\ast)$ is not fixed bounded. Indeed, recalling the discussion on \cite[Sect. II.22]{Treves-book}, and keeping in mind that $S(\Fg^\ast)$ is the space of polynomials on $\Fg$, it suffices to consider the seminorms
\[
\rho_\lambda\left(\sum_{(n_1,\ldots,n_s) \in \Nb^n}\a_{(n_1,\ldots,n_s)}X_1^{n_1}\ldots X_s^{n_s}\right) := \sup_{n_1+\ldots +n_s < \lambda} |\a_{(n_1,\ldots,n_s)}|,\qquad \lambda = 0,1,2,\ldots
\]
regarding $\Fg = \langle X_1,\ldots, X_s \rangle$, and $\a_{(n_1,\ldots,n_s)} = 0$ for all but finitely many $(n_1,\ldots,n_s) \in \Nb^n$. Nevertheless, it is possible to walk around the problem of $S(\Fg^\ast)$ being not fixed bounded, focusing instead its being locally multiplicatively convex.

\ni A topological algebra is called locally multiplicatively convex, if its topology can be defined by a family of submultiplicative\footnote[1]{A semi-norm $\rho:A \lra [0,\infty)$ is said to be submultiplicative if $\rho(ab)\leq \rho(a)\rho(b)$ for any $a,b \in A$.} semi-norms. In order to check whether an algebra is locally multiplicatively convex, one uses the following lemma of \cite{MitiRoleZela6162}, see also \cite[Coroll. 1.3]{Pirk06-II}.

\begin{lemma}\label{lemma-A-loc-multp-convex}
Let $A$ be a locally convex topological algebra. Suppose $A$ has a base $\mathscr{U}$ of absolutely convex\footnote[2]{Convex and balanced, where a subset $\Uc$ of a l.c. t.v.s. $W$ is called balanced if $\lambda w \in \Uc$ for any $w \in \Uc$ and any $\lambda \in \mathbb{C}$ with $|\lambda|\leq 1$.} neighborhoods at 0 with the property that for each $V \in \mathscr{U}$ there exists $U \in \mathscr{U}$ and $C>0$ such that $UV \subseteq CV$. Then $A$ is locally multiplicatively convex.
\end{lemma}

\ni We now see that this is the case for $U(\Fg)$, equipped with the natural topology.

\begin{lemma}\label{lemma-loc-multp-convex-U}
For a topological Lie algebra $\Fg$, the universal enveloping algebra $U(\Fg)$ is locally multiplicatively convex.
\end{lemma}

\begin{proof}
The canonical filtration $(U_n(\Fg))_{n\in \Nb}$, see for instance \cite[Subsect. 2.3.1]{Dixm-book}, determines a sequence of definition for the natural topology on $U(\Fg)$.

\ni Let $V$ be an absolutely convex neighbourhood of $0$. Let also 
\begin{equation*}
U:=\{\a1 \mid \a\in \Cb,\,\,|\a|\leq 1\} \subseteq U_0(\Fg).
\end{equation*}
It follows immediately that $U\cap U_n(\Fg) = U$ is convex for any $n\in \Nb$, and that $U$ is balanced. Moreover, since $V$ is balanced, $UV \subseteq V$, and hence the claim.
\end{proof}

\ni The same is true for the symmetric algebra $S(\Fg^*)$.

\begin{corollary}\label{lemma-loc-multp-convex-S}
For a topological Lie algebra $\Fg$, the symmetric algebra $S(\Fg^*)$ is locally multiplicatively convex.
\end{corollary}

\begin{proposition}
Let $\Fg$ be a finite dimensional Lie algebra. Then the symmetric algebra $S(\Fg^\ast)$, equipped with the Koszul coaction \eqref{Koszul-coact}, is a topological $U(\Fg)$-comodule via \eqref{aux-g-to-U(g)-comod}.
\end{proposition}

\begin{proof}
Recalling the element $\sum_{i\in I} f^i\,\projot\,X_i \in \Fg^\circ \,\projot\,\Fg \cong \End(\Fg)$, we may express the Koszul coaction \eqref{Koszul-coact} as $\nb_\Fg(R)=\sum_{i \in I}X_i\,\projot\, Rf^i$ for any $R\in S(\Fg^\ast)$. As a result, we have
\begin{equation*}
\nb_\Fg^{\exp}(R) = \sum_{k=0}^\infty\frac{1}{k!}X_{i_k}\ldots X_{i_1} \,\projot\, Rf^{i_1}\cdots f^{i_k}.
\end{equation*}
Then, along the lines of the proof of Proposition \ref{prop-g-coact-to-Ug-coact}, for
\begin{equation}\label{aux-sn-for-Sg-ast}
s_n := \sum_{k=0}^n\frac{1}{k!}X_{i_k}\ldots X_{i_1} \,\projot\, Rf^{i_1}\cdots f^{i_k},
\end{equation}
and any seminorm $\rho_{p,q}$ on $U(\Fg) \,\projot\,S(\Fg^\ast)$, we observe that
\begin{align*}
\begin{split}
& \rho_{p,q}(s_n-s_m) \leq \sum_{k=m+1}^n\frac{1}{k!}p(X_{i_k}\ldots X_{i_1})q(Rf^{i_1}\cdots f^{i_k}) \\
& \leq \sum_{k=m+1}^n\frac{1}{k!}p(X_{i_k})\ldots p(X_{i_1})q(R)q(f^{i_1})\ldots q(f^{i_k}),
\end{split}
\end{align*}
since $U(\Fg)$ and $S(\Fg^\ast)$ are locally multiplicatively convex by Lemma \ref{lemma-loc-multp-convex-U} and Corollary \ref{lemma-loc-multp-convex-S}. It then follows from \cite[Prop. 7.7]{Treves-book} that there is a semi-norm $\mu$ on $S(\Fg^\ast)$ such that 
\begin{equation*}
\sum_{i_j}p(X_{i_j})q(f^{i_j})\leq \mu(1).
\end{equation*}
As a result,
\begin{equation*}
\rho_{p,q}(s_n-s_m) \leq \sum_{k=m+1}^n\frac{1}{k!}\mu(1)^kq(R).
\end{equation*}
Therefore, the sequence $(s_n)_{n\in \Nb}$ of partial sums given by \eqref{aux-sn-for-Sg-ast} is Cauchy, and hence converges in $U(\Fg)\,\projot\,S(\Fg^\ast)$.
\end{proof}

\ni We conclude this subsection with a short discussion on AYD modules over topological Lie algebras.

\begin{definition}
Let $V$ be a topological right module / left comodule over a topological Lie algebra $\Fg$. We call $V$ a right-left topological AYD module over $\Fg$ if
\begin{equation}
\nb(v\cdot X)=v\nsb{-1}\,\projot\, v\nsb{0}\cdot X + [v\nsb{-1},X]\,\projot\, v\nsb{0}.
\end{equation}
In addition, $V$ is called stable if
\begin{equation}
v\nsb{0}\cdot v\nsb{-1}=0.
\end{equation}
\end{definition}

\begin{proposition}\label{prop-g-AYD-to-Ug-AYD}
Let $\Fg$ be a topological Lie algebra, and $V$ a topological $\Fg$-module / comodule. Then $V$ is a topological right-left AYD module over $\Fg$ if and only if it is a topological right-left AYD module over $U(\Fg)$.
\end{proposition}

\begin{proof}
It follows from Proposition \ref{prop-Ug-coact-to-g-coact} that if $V$ is a left $U(\Fg)$-comodule, then it is a left $\Fg$-comodule. On the next step we use Proposition \ref{prop-g-coact-to-Ug-coact} to conclude that $V$ is a left $U(\Fg)$-comodule if it is a $\Fg$-comodule. Finally, the AYD condition is similar to \cite[Lemma 5.10]{RangSutl-II}.
 \end{proof}

\subsection{Cyclic cohomology theories for topological Lie algebras}\label{subsect-cyclic-complexes-lie-alg}

\ni In this subsection we extend to topological Lie algebras, the two cyclic complexes introduced in \cite{RangSutl-II}. More precisely, the complex associated to a Lie algebra and a SAYD module over it, generalizing the Lie algebra homology complex, and the one associated to a Lie algebra and a unimodular SAYD module over it, generalizing the Lie algebra cohomology complex.


\ni We shall begin with the unimodular stability \cite[Prop. 2.4]{RangSutl-II} on the level of topological Lie algebras.

\begin{definition}
Let $\Fg$ be a topological Lie algebra, and $V$ a topological right $\Fg$-module / left $\Fg$-comodule. Then $V$ is called unimodular stable over $\Fg$ if any element of $V$ is annihilated by $\sum_{i\in I} X_i\,\projot\,f^i \in \Fg\,\projot\,\Fg^\circ \cong \End(\Fg)$, that is,
\begin{equation}
v\cdot \left(\sum_{i\in I} X_i\,\projot\,f^i\right) = 0.
\end{equation}
\end{definition}

\ni If $\Fg$ is finite dimensional, then $\sum_{i\in I} X_i\,\projot\,f^i \in \Fg^\ast\,\ot\,\Fg$ consists of a dual pair of bases, and the above definition coincides with the one given in \cite{RangSutl-II}. Similarly, the stability can be given by the trivial action of $\sum_{i\in I} f^i\,\projot\,X_i\in \Fg^\circ\,\projot\,\Fg$.

\medskip

\ni Let us next recall the cohomology of the topological Lie algebras from \cite[Def. I.1]{Neeb06}, see also \cite{Tsuj81}. Given a topological Lie algebra $\Fg$, and a topological $\Fg$-module $V$, let $W_{\top}^n(\Fg,V)$ denote the space of continuous alternating maps $\Fg^{\times\,n}\lra V$ for $n\geq 0$. In other words, $W_{\top}^n(\Fg,V)$ is the set of continuous $n$-cochains with values in $V$. Then, $W_{\top}^n(\Fg,V)$ is a differential complex with the differential
\begin{align*}
\begin{split}
& d_{\rm CE}:W_{\top}^n(\Fg,V) \lra W_{\top}^{n+1}(\Fg,V),\\
& d_{\rm CE}(\a)(Y_0,\ldots, Y_n):=\sum_{i<j}(-1)^{i+j}\a([Y_i,Y_j],Y_0,\ldots,\widehat{Y}_i,\ldots,\widehat{Y}_j,\ldots Y_n) \\
& + \sum_{j=0}^n(-1)^{j+1}\a(Y_0,\ldots,\widehat{Y}_j,\ldots, Y_n)\cdot Y_j.
\end{split}
\end{align*}

\ni Similarly we define the Koszul boundary map, see \cite[Sect. 3.2]{RangSutl-II}, by
\begin{align*}
\begin{split}
& d_{\rm K}:W_{\top}^{n+1}(\Fg,V) \lra W_{\top}^n(\Fg,V),\\
& d_{\rm K}(\b)(Y_1,\ldots, Y_n):=\sum_{i\in I}\iota_{X_i}(\b)(Y_1,\ldots, Y_n)\cdot f^i = \sum_{i\in I}\b(X_i,Y_1,\ldots, Y_n)\cdot f^i.
\end{split}
\end{align*}
Since the truncation and the $S(\Fg^\ast)$-action are continuous, by Proposition \ref{prop-g-comod-Sg-mod}, the Koszul boundary map restricts to the continuous cochains. Moreover, for any $\b\in W_{\top}^{n+1}(\Fg,V)$,
\begin{equation*}
d_{\rm K}^2(\b)(Y_1,\ldots,Y_{n-1}) = \sum_{i,j\in I}\iota_{X_i}\iota_{X_j}(\b)(Y_1,\ldots,Y_{n-1})\cdot f^jf^i=0,
\end{equation*}
that is, $d_{\rm K}^2=0$.

\medskip

\ni We are now ready to define a cyclic cohomology theory for topological Lie algebras. To this end we record the following analogue of \cite[Thm. 2.4]{RangSutl-II}.

\begin{proposition}
Given a topological Lie algebra $\Fg$, and a topological (right) $\Fg$-module / (left) $\Fg$-comodule $V$,
\begin{equation}\label{aux-W-cont}
(W^\ast_{\top}(\Fg,V),d_{\rm CE}+d_{\rm K}), \qquad W^\ast_{\top}(\Fg,V):=\bigoplus_{n\geq 0}W_{\top}^n(\Fg,V)
\end{equation}
is a differential complex if and only if $V$ is a unimodular SAYD module over $\Fg$.
\end{proposition}

\begin{proof}
Given any $\g \in W_{\top}^n(\Fg,V)$ for $n\geq 0$, we have
\begin{align*}
\begin{split}
& d_{\rm CE}d_{\rm K}(\g)(Y_0,\ldots,Y_{n-1})=\\
& \sum_{i<j}(-1)^{i+j}d_{\rm K}(\g)([Y_i,Y_j],Y_0,\ldots,\widehat{Y}_i,\ldots,\widehat{Y}_j,\ldots Y_{n-1}) \\
& + \sum_{j=0}^{n-1}(-1)^{j+1}d_{\rm K}(\g)(Y_0,\ldots,\widehat{Y}_j,\ldots, Y_{n-1})\cdot Y_j \\
& =\sum_{k\in I}\sum_{i<j}(-1)^{i+j}\g(X_k,[Y_i,Y_j],Y_0,\ldots,\widehat{Y}_i,\ldots,\widehat{Y}_j,\ldots ,Y_{n-1})\cdot f^k \\
& + \sum_{k\in I}\sum_{j=0}^{n-1}(-1)^{j+1}(\g(X_k,Y_0,\ldots,\widehat{Y}_j,\ldots, Y_{n-1})\cdot f^k)\cdot Y_j,
\end{split}
\end{align*}
as well as
\begin{align*}
\begin{split}
& d_{\rm K}d_{\rm CE}(\g)(Y_0,\ldots,Y_{n-1})= \sum_{k\in I}d_{\rm CE}(\g)(X_k,Y_0,\ldots,Y_{n-1})\cdot f^k \\
& = \sum_{k\in I}\sum_{i<j}(-1)^{i+j}\g([Y_i,Y_j],X_k,Y_0,\ldots,\widehat{Y}_i,\ldots,\widehat{Y}_j,\ldots, Y_{n-1})\cdot f^k\\
& + \sum_{k\in I}\sum_{j=0}^{n-1}(-1)^{j+1}\g([X_k,Y_j],Y_0,\ldots,\widehat{Y}_j,\ldots,Y_{n-1})\cdot f^k\\
& - \sum_{k\in I}(\g(Y_0,\ldots,Y_{n-1})\cdot X_k)\cdot f^k\\
& +\sum_{k\in I}\sum_{j=0}^{n-1}(-1)^{j}(\g(X_k,Y_0,\ldots,\widehat{Y}_j,\ldots,Y_{n-1})\cdot Y_j)\cdot f^k.
\end{split}
\end{align*}
As a result, \eqref{aux-W-cont} is a differential complex, \ie $d_{\rm CE}d_{\rm K}+d_{\rm K}d_{\rm CE}=0$ if and only if for any $\g\in W_{\top}^n(\Fg,V)$, and any $Y_0,\ldots,Y_{n-1}\in\Fg$,
\begin{align}\label{diffble-complex-condition}
\begin{split}
& \sum_{k\in I}\sum_{j=0}^{n-1}(-1)^{j+1}\big[(\g(X_k,Y_0,\ldots,\widehat{Y}_j,\ldots, Y_{n-1})\cdot f^k)\cdot Y_j  \\
&\hspace{2cm} -\sum_{k\in I}(\g(X_k,Y_0,\ldots,\widehat{Y}_j,\ldots,Y_{n-1})\cdot Y_j)\cdot f^k\big]\\
&+ \sum_{k\in I}\sum_{j=0}^{n-1}(-1)^{j+1}\g([X_k,Y_j],Y_0,\ldots,\widehat{Y}_j,\ldots,Y_{n-1})\cdot f^k\\
&- \sum_{k\in I}(\g(Y_0,\ldots,Y_{n-1})\cdot X_k)\cdot f^k=0.
\end{split}
\end{align}
Now, if $W_{\top}(\Fg,V)$ is a differential complex, then \eqref{diffble-complex-condition} for a $\g\in W^0_{\top}(\Fg,V) := V$ yields
\begin{equation}\label{unimodular-stability}
\sum_{k\in I}(\g\cdot X_k)\cdot f^k =0,
\end{equation}
that is, the unimodular stability, and for a $\g\in W^1_{\top}(\Fg,V)$, provides
\begin{align}\notag
& \sum_{k\in I}(\g(X_k)\cdot Y)\cdot f^k-\sum_{k\in I}(\g(X_k)\cdot f^k)\cdot Y \\\notag
& \hspace{3cm}-\sum_{k\in I}\g([X_k,Y])\cdot f^k-\sum_{k\in I}(\g(Y)\cdot X_k)\cdot f^k =\\\label{AYD-funct}
& \sum_{k\in I}(\g(X_k)\cdot Y)\cdot f^k- \sum_{k\in I}(\g(X_k)\cdot f^k)\cdot Y-\sum_{k\in I}\g([X_k,Y])\cdot f^k=0,
\end{align}
which is precisely the AYD compatibility. Conversely, if $V$ is a unimodular stable AYD module over $\Fg$, then we already have \eqref{unimodular-stability}, the unimodular stability, in view of which \eqref{diffble-complex-condition} reduces to 
\begin{align}\label{diffble-complex-condition-reduced}
\begin{split}
& \sum_{k\in I}\sum_{j=0}^{n-1}(-1)^{j+1}\big[(\g(X_k,Y_0,\ldots,\widehat{Y}_j,\ldots, Y_{n-1})\cdot f^k)\cdot Y_j  \\
&\hspace{2cm} -\sum_{k\in I}(\g(X_k,Y_0,\ldots,\widehat{Y}_j,\ldots,Y_{n-1})\cdot Y_j)\cdot f^k\big]\\
&+ \sum_{k\in I}\sum_{j=0}^{n-1}(-1)^{j+1}\g([X_k,Y_j],Y_0,\ldots,\widehat{Y}_j,\ldots,Y_{n-1})\cdot f^k = 0,
\end{split}
\end{align}
and \eqref{AYD-funct} for any $\g\in W^1_{\top}(\Fg,V)$. Given $\g\in W^n_{\top}(\Fg,V)$, and $Z_1,\ldots, Z_{n-1} \in \Fg$, setting
\[
\widetilde{\g} \in W^1_{\top}(\Fg,V), \qquad \widetilde{\g}(X) := \g(X,Z_1,\ldots, Z_{n-1}),
\]
we see that \eqref{AYD-funct}  implies \eqref{diffble-complex-condition-reduced}.
\end{proof}

\begin{definition}
The cyclic cohomology of $\Fg$ with coefficients in a unimodular SAYD module $V$ over $\Fg$, which is  denoted by $\widetilde{HC}^\ast_{\top}(\Fg,V)$, is defined to be the total cohomology of the bicomplex
\begin{align*}
W_{\top}^{p,q}(\Fg,V)= \left\{ \begin{matrix} W_{\top}^{q-p}(\Fg,V),&\quad\text{if}\quad 0\le p\le q, \\
&\\
0, & \text{otherwise.}
\end{matrix}\right.
\end{align*}
Similarly, the periodic cyclic cohomology of $\Fg$ with coefficients in a unimodular SAYD module $V$ over $\Fg$, which is denoted by $\widetilde{HP}^\ast_{\top}(\Fg,V)$, is defined to be the total cohomology of
\begin{align*}
W_{\top}^{p,q}(\Fg,V)= \left\{ \begin{matrix} W_{\top}^{q-p}(\Fg,V),&\quad\text{if}\quad  p\le q, \\
&\\
0, & \text{otherwise.}
\end{matrix}\right.
\end{align*}
\end{definition}

\ni Let $\Fg$ be a topological Lie algebra, and $V$ a topological right $\Fg$-module / left $\Fg$-comodule. Similar to \cite[Sect. 4.2]{RangSutl-II}, we can associate a cyclic cohomology theory to a topological Lie algebra $\Fg$, and a SAYD module over it, by setting
\begin{equation*}
C_n^{\top}(\Fg,V):=\wg^n\Fg\,\projot\,V
\end{equation*}
with two continuous differentials
\begin{align*}
\begin{split}
& \part_{\rm CE}: C_{n+1}^{\top}(\Fg,V) \lra C_n^{\top}(\Fg,V),\\
& \part_{\rm CE}(Y_0\wdots Y_n\,\projot\,v) = \sum_{j=0}^n(-1)^{j+1} Y_0\wdots \widehat{Y}_j \wdots Y_n\,\projot\,v\cdot Y_j + \\
& \sum_{j,k=0}^n(-1)^{j+k} [Y_j,Y_k]\wg Y_0\wdots \widehat{Y}_j \wdots \widehat{Y}_k \wdots Y_n\,\projot\,v,
\end{split}
\end{align*}
and
\begin{align*}
\begin{split}
& \part_{\rm K}: C_n^{\top}(\Fg,V) \lra C_{n+1}^{\top}(\Fg,V),\\
& \part_{\rm K}(Y_1\wdots Y_n\,\projot\,v) = v\nsb{-1}\wg Y_1\wdots Y_n\,\projot\,v\nsb{0}.
\end{split}
\end{align*}
Having defined the (co)boundary maps, we can immediately state the following analogue of \cite[Prop. 4.2]{RangSutl-II}.

\begin{proposition}
Let $\Fg$ be a topological Lie algebra, and $V$ a topological right $\Fg$-module / left $\Fg$-comodule. Then $(C^{\top}_\ast(\Fg,V),\,\part_{\rm CE}+\part_{\rm K})$ is a differential complex if and only if $V$ is a SAYD module over $\Fg$.
\end{proposition}

\ni As a result, we define the cyclic cohomology of a topological Lie algebra, and a topological SAYD module associated to it.

\begin{definition}
The cyclic cohomology of $\Fg$ with coefficients in a SAYD module $V$ over $\Fg$, which is  denoted by $HC^\ast_{\top}(\Fg,V)$, is defined to be the total cohomology of the bicomplex
\begin{align*}
C^{\top}_{p,q}(\Fg,V)= \left\{ \begin{matrix} C^{\top}_{q-p}(\Fg,V),&\quad\text{if}\quad 0\le p\le q, \\
&\\
0, & \text{otherwise.}
\end{matrix}\right.
\end{align*}
Similarly, the periodic cyclic cohomology of $\Fg$ with coefficients in a SAYD module $V$ over $\Fg$, which is denoted by $HP^\ast_{\top}(\Fg,V)$, is defined to be the total cohomology of
\begin{align*}
C^{\top}_{p,q}(\Fg,V)= \left\{ \begin{matrix} C^{\top}_{q-p}(\Fg,V),&\quad\text{if}\quad  p\le q, \\
&\\
0, & \text{otherwise.}
\end{matrix}\right.
\end{align*}
\end{definition}

\section{From Hopf-cyclic to Lie cyclic}\label{sect-comput}

\ni In this section we identify the Hopf-cyclic cohomology of the Hopf algebra $\Fc^\infty(G_2)\tacl U(\Fg_1)$, associated to a matched pair $(G_1,G_2)$ of Lie groups, with the cyclic cohomology of the Lie algebra $\Fg_1 \bowtie \Fg_2$, associated to the (matched) pair of Lie algebras $(\Fg_1,\Fg_2)$ of the Lie groups $(G_1,G_2)$. To this end, we consider in the first subsection the coalgebra Hochschild cohomology of $\Fc^\infty(G)$, which appears in the $E_1$-page of a spectral sequence computing the Hopf-cyclic cohomology of $\Fc^\infty(G_2)\tacl U(\Fg_1)$.

\subsection{Coalgebra Hochschild cohomology of $\Fc^\infty(G)$}\label{subsect-comput-Hochschild-cohom}

In this subsection we identify the coalgebra Hochschild cohomology of the Hopf algebra $\Fc^\infty(G)$ of differentiable functions on a real analytic group $G$,  with the differentiable group cohomology (and hence the continuous cohomology by \cite{HochMost62}) of the group $G$. The latter, in turn, is identified with the relative Lie algebra cohomology of the Lie algebra $\Fg$ of $G$, relative to the Lie algebra of a maximal compact subgroup of $G$, \cite{vanEst53,Dupo76,HochMost62}.

\ni Let us first recall the differentiable cohomology from \cite{HochMost62}. Keeping Definition \ref{def-diffble-module}  and Definition \ref{def-diffble-map} (of a differentiable $G$-module, and of a differentiable map, respectively) in mind, given a differentiable $G$-module $V$, let $C^n_d(G,V)$ be the space of all differentiable maps $G^{\times\,n}\to V$, and
\begin{equation*}
C^\ast_d(G,V) = \bigoplus_{n\geq 0}C_d^n(G,V),
\end{equation*}
together with the differential map $d: C_d^n(G,V) \to C_d^{n+1}(G,V)$ which is given by
\begin{align}\label{aux-non-hom-cobound}
\begin{split}
& d(c)(g_1,\ldots,g_{n+1}):= g_1\cdot c(g_2,\ldots,g_{n+1}) +\\
&\hspace{2cm} \sum_{i=1}^n(-1)^ic(g_1,\ldots,g_ig_{i+1},\ldots,g_{n+1}) + (-1)^{n+1}c(g_1,\ldots,g_n).\\
\end{split}
\end{align}
Then $(C^\ast_d(G,V),d)$ is a differential graded complex whose cohomology is called the differentiable cohomology of $G$ with coefficients in the differentiable $G$-module $V$, and is denoted by $H^\ast_d(G,V)$.

\begin{proposition}\label{porp-diff-cohom}
Let $V$ be a left $\Fc^\infty(G)$-comodule via $v\mapsto v^{\ns{-1}}\,\projot\, v^{\ns{0}}$. Then the map
\begin{align*}
\begin{split}
& \Psi:C^n(\Fc^\infty(G),V)\lra C_d^n(G,V),\\
& \Psi(v\,\projot\,f^1\,\projot\,\ldots \,\projot\, f^n)(g_1,\ldots,g_n) := vf^1(g_n^{-1})f^2(g_{n-1}^{-1})\,\ldots\,f^n(g_1^{-1})
\end{split}
\end{align*}
is an isomorphism of complexes.
\end{proposition}

\begin{proof}
We note by Proposition \ref{prop-diff-ble-G-module} that being a left $\Fc^\infty(G)$-comodule, $V$ is a differentiable right $G$-module. Thus, $V$ possesses the left $G$-module structure given by $x\cdot v:=v\cdot x^{-1}$. Accordingly we have
\begin{align}\label{aux-d-Psi}
\begin{split}
& d\left(\Psi(v\,\projot\,f^1\,\projot\,\ldots \,\projot\, f^n)\right)(g_1,\ldots,g_{n+1}) = \\
& \hspace{4cm} g_1 \cdot \Psi(v\,\projot\,f^1\,\projot\,\ldots \,\projot\, f^n) (g_2,\ldots, g_{n+1}) + \\
& \sum_{j=1}^n (-1)^j \Psi(v\,\projot\,f^1\,\projot\,\ldots \,\projot\, f^n)(g_1,\ldots,g_jg_{j+1},\ldots,g_{n+1}) + \\
& \hspace{4cm} \Psi(v\,\projot\,f^1\,\projot\,\ldots \,\projot\, f^n) (g_1,\ldots, g_n) = \\
& g_1\cdot v f^1(g_{n+1}^{-1}) \ldots f^n(g_2^{-1}) + \\
& \sum_{j=1}^{n} (-1)^{j} vf^1(g_{n+1}^{-1}) \cdots f^{n-j+1}((g_jg_{j+1})^{-1})\cdot f^n(g_1^{-1}) + \\
& (-1)^{n+1} vf^1(g_n^{-1}) \cdot f^n(g_1^{-1}),
\end{split}
\end{align}
and
\begin{align*}
\begin{split}
& \Psi(b\left(v\,\projot\,f^1\,\projot\,\ldots \,\projot\, f^n\right))(g_1,\ldots,g_{n+1}) = \\
& \Psi(v\,\projot\, \ve\,\projot\,f^1\,\projot\,\ldots \,\projot\, f^n)(g_1,\ldots,g_{n+1}) + \\
& \sum_{j=1}^n (-1)^j \Psi(v\,\projot\,f^1\,\projot\,\ldots \,\projot\, \D(f^j)\,\projot\,\ldots \,\projot\, f^n)(g_1,\ldots,g_{n+1}) +\\
& (-1)^{n+1} \Psi(v^{\ns{0}}\,\projot\,f^1\,\projot\,\ldots \,\projot\, f^n\,\projot\, v^{\ns{-1}})(g_1,\ldots,g_{n+1}) =\\
& vf^1(g_n^{-1})\ldots f^n(g_1^{-1}) + \\
& \sum_{j=1}^n (-1)^j v f^1(g_{n+1}^{-1}) \ldots f^j(g_{n-j+2}^{-1}g_{n-j+1}^{-1}) \ldots f^n(g_1^{-1}) + \\
& (-1)^{n+1} g_1\cdot v f^1(g_{n+1}^{-1})\ldots f^n(g_1^{-1}).
\end{split}
\end{align*}
On the other hand, for $1 \leq j\leq n$, one has $1\leq k:=n+1-j \leq n$, and by this substitution we rewrite \eqref{aux-d-Psi} as
\begin{align*}
\begin{split}
& d\left(\Psi(v\,\projot\,f^1\,\projot\,\ldots \,\projot\, f^n)\right)(g_1,\ldots,g_{n+1}) = \\
& g_1\cdot v f^1(g_{n+1}^{-1}) \ldots f^n(g_2^{-1}) + \\
& \sum_{k=1}^{n} (-1)^{n+1-k} vf^1(g_{n+1}^{-1}) \ldots f^{k}((g_{n+1-k}g_{n+2-k})^{-1})\cdot f^n(g_1^{-1}) + \\
& (-1)^{n+1} vf^1(g_n^{-1}) \cdot f^n(g_1^{-1}),
\end{split}
\end{align*}
As a result, we see that
\begin{equation*}
d\left(\Psi(v\,\projot\,f^1\,\projot\,\ldots \,\projot\, f^n)\right) = (-1)^{n+1} \Psi(b\left(v\,\projot\,f^1\,\projot\,\ldots \,\projot\, f^n\right)).
\end{equation*}
\end{proof}

\ni Following \cite{Most61,HochMost62} we denote by $H^\ast_c(G,V)$ the continuous cohomology of a topological group $G$, with coefficients in a continuous (left) $G$-module. A continuous $G$-module is defined to be a Hausdorff t.v.s. such that the $G$-action $G\times V \lra V$ is continuous, and $H^\ast_c(G,V)$ is defined to be the homology with respect to the coboundary \eqref{aux-non-hom-cobound}.

\ni Let us next recall the relation between the continuous cohomology and the differentiable cohomology from \cite[Thm. 5.1]{HochMost62}. To this end, we need to recall the integrability of the coefficient space \cite{HochMost62}, see also \cite[2.13]{Most61}.

\begin{definition}
Let $V$ be a t.v.s., and $G$ a locally compact topological group. Let also $F(G,V)$ be the space of all continuous maps $G\lra V$ with compact support, topologized by the compact-open topology. Then $V$ is called $G$-integrable if there is a continuous map $J_G:F(G,V) \lra V$ such that for any $\g\in V^\circ$ and any $f\in F(G,V)$,
\begin{equation}
\g(J_G(f)) = I_G(\g\circ f),
\end{equation}
where $I_G$ is a Haar integral on $G$.
\end{definition}

\ni In case $G$ is replaced by $\Rb$, and the Haar integral by an ordinary integral over $[0,1]$, $V$ is called $[0,1]$-integrable. The following is \cite[Thm. 5.1]{HochMost62}.

\begin{theorem}\label{thm-Hd-to-Hc}
Let $G$ be a real analytic group, and $V$ a differentiable (left) $G$-module. If $V$ is locally convex and $G$-integrable, then the canonical homomorphism $H^\ast_d(G,V)\lra H^\ast_c(G,V)$ is an isomorphism.
\end{theorem}

\ni We note from \cite[2.13(iii)]{Most61} that any finite dimensional Hausdorff, or more generally any continuous $G$-module which is complete as a t.v.s. is $G$-integrable, see \cite[Sect. 6]{HochMost62}.

\ni The passage to the continuous group cohomology enables us to link the Hochschild cohomology of $\Fc^\infty(G)$ to the (relative) Lie algebra cohomology, \cite[Sect. 6]{HochMost62}. 

\begin{theorem}\label{thm-Hc-to-HLie}
Let $G$ be a real analytic group, and $K$ be a maximal compact subgroup of $G$. Let also $V$ be a locally convex $G$-, $K$-, and $[0,1]$-integrable differentiable $G$-module. Then,
\begin{equation}
H^\ast_c(G,V) \cong H^\ast_{CE}(\Fg,\Fk,V).
\end{equation}
\end{theorem}


\subsection{Hopf-cyclic cohomology of  $\Fc^\infty(G_2)\tacl U(\Fg_1)$}\label{subsect-comput-Lie-Hopf}

\ni In this subsection we shall identify the (periodic) Hopf-cyclic cohomology of the Hopf algebra $\Fc^\infty(G_2)\tacl U(\Fg_1)$ with the Lie algebra cohomology of $\Fg_1\bi \Fg_2 $ relative to $\Fk$, the maximal compact subalgebra of $\Fg_2$, via a van Est type isomorphism. 

\medskip

\begin{theorem}\label{thm-main-identf}
Let $(G_1,G_2)$ be a matched pair of Lie groups, with Lie algebras $(\Fg_1,\Fg_2)$. Then the periodic Hopf-cyclic cohomology of the Hopf algebra $\Fc^\infty(G_2)\tacl U(\Fg_1)$  with coefficients in the canonical MPI ${}^\s k_\d$ is isomorphic with the Lie algebra  cohomology, with trivial coefficients, of $\Fg_1\bi\Fg_2$ relative to the maximal compact Lie subalgebra $\Fk$ of $\Fg_2$. In short,
\begin{equation*}
HP^\ast_{top}(\Fc^\infty(G_2)\tacl U(\Fg_1),{}^\s k_\d) \cong \widetilde{HP}^\ast(\Fg_1\bi\Fg_2,\Fk).
\end{equation*}
\end{theorem}

\begin{proof}
Let us first recall from \cite{Tayl72} that $E,F,G$ being topological spaces with $G$ being complete, a jointly continuous map\footnote[1]{$E,F,G$ being topological spaces; a map $E\times F \to G$ is called ``jointly continuous'' if it is continuous with respect to the product topology on $E \times F$.} $E\times F \to G$ extends (uniquely) to $E \,\projot\, F\to G$. As such, equipped with the natural topology, hence the bilinear maps being continuous \cite[Lemma A.2.8]{BonnFlatGersPinc94}, the isomorphisms of \cite[Prop. 3.16]{MoscRang09}, see also \cite[Prop. 5.1]{RangSutl} and \cite[Thm. 4.6]{RangSutl-III}, identifying the Hopf-cyclic cohomology of a bicrossed product Hopf algebra with the (total) cohomology of a bicomplex, extends to the completed projective tensor product. Therefore, the Hopf-cyclic cohomology of the bicrossed product Hopf algebra $\Fc^\infty(G_2)\tacl U(\Fg_1)$ may be computed by the bicomplex
\begin{align}\label{aux-bicomplex-H(G)-l}
\begin{xy} \xymatrix{  \vdots & \vdots
 &\vdots &&\\
 \wg^2{{\Fg_1}}^\ast  \ar[u]^{d_{\rm CE}}\ar[r]^{b\;\;\;\;\;\;\;\;\;\;}&   \wg^2{{\Fg_1}}^\ast\,\projot\, \Fc^\infty(G_2) \ar[u]^{d_{\rm CE}} \ar[r]^b &  \wg^2{{\Fg_1}}^\ast\,\projot\, \Fc^\infty(G_2)^{\ot 2} \ar[u]^{d_{\rm CE}} \ar[r]^{\;\;\;\;\;\;\;\;\;\;\;\;\;\;b} & \hdots&  \\
  {{\Fg_1}}^\ast  \ar[u]^{d_{\rm CE}}\ar[r]^{b\;\;\;\;\;\;\;\;\;\;\;\;\;}&  {{\Fg_1}}^\ast\,\projot\, \Fc^\infty(G_2) \ar[u]^{d_{\rm CE}} \ar[r]^b& {{\Fg_1}}^\ast\,\projot\, \Fc^\infty(G_2)^{\ot 2} \ar[u]^{d_{\rm CE}} \ar[r]^{\;\;\;\;\;\;\;\;\;\;\;\;b }& \hdots&  \\
  k \ar[u]^{d_{\rm CE}}\ar[r]^{b~~~~~~~}&  \Fc^\infty(G_2) \ar[u]^{d_{\rm CE}}\ar[r]^b& \Fc^\infty(G_2)^{\ot 2} \ar[u]^{d_{\rm CE}} \ar[r]^{\;\;\;\;\;\;\;b} & \hdots&, }
\end{xy}
\end{align}
where $d_{\rm CE}:\wg^q{\Fg_1}^\ast\,\projot\, \Fc^\infty(G_2)^{\ot p}\to \wg^{q+1}{\Fg_1}^\ast\,\projot\, \Fc^\infty(G_2)^{\ot p}$ is the Lie algebra cohomology of ${\Fg_1}$, with  coefficients in  $\Fc^\infty(G_2)^{\ot p}$. Similarly, $b:\wg^q{\Fg_1}^\ast\,\projot\, \Fc^\infty(G_2)^{\ot p} \to \wg^q{\Fg_1}^\ast\,\projot\, \Fc^\infty(G_2)^{\ot p+1}$ is the coalgebra Hochschild cohomology  coboundary with coefficients in the $\Fc^\infty(G_2)$-comodule $\wg^q{{\Fg_1}}^\ast$.
Since $\wg^q{\Fg_1}^\ast$ is a differentiable $G_2$-module,  by Proposition \ref{prop-diff-ble-G-module} it is  an $\Fc^\infty(G_2)$-comodule. As a result, by Proposition \ref{porp-diff-cohom} the bicomplex \eqref{aux-bicomplex-H(G)-l} is isomorphic with the bicomplex
\begin{align*}
\begin{xy} \xymatrix{  \vdots & \vdots
 &\vdots &&\\
 C_d^0(G_2,\wg^2{{\Fg_1}}^\ast)  \ar[u]^{d_{\rm CE}}\ar[r]^{{d^\ast}\;\;}&   C_d^1(G_2,\wg^2{{\Fg_1}}^\ast) \ar[u]^{d_{\rm CE}} \ar[r]^{{d^\ast}}&  C_d^2(G_2,\wg^2{{\Fg_1}}^\ast) \ar[u]^{d_{\rm CE}} \ar[r]^{\;\;\;\;\;\;\;\;\;\;{d^\ast}} & \hdots&  \\
  C_d^0(G_2,{{\Fg_1}}^\ast)  \ar[u]^{d_{\rm CE}}\ar[r]^{{d^\ast}\;\;}&  C_d^1(G_2,{{\Fg_1}}^\ast) \ar[u]^{d_{\rm CE}} \ar[r]^{{d^\ast}}& C_d^2(G_2,{{\Fg_1}}^\ast) \ar[u]^{d_{\rm CE}} \ar[r]^{\;\;\;\;\;\;\;\;\;\;\;{d^\ast} }& \hdots&  \\
  C_d^0(G_2,k) \ar[u]^{d_{\rm CE}}\ar[r]^{{d^\ast}~}&  C_d^1(G_2,k) \ar[u]^{d_{\rm CE}}\ar[r]^{{d^\ast}}& C_d^2(G_2,k) \ar[u]^{d_{\rm CE}} \ar[r]^{\;\;\;\;\;\;\;\;\;{d^\ast}} & \hdots&, }
\end{xy}
\end{align*}
where $d^\ast:C_d^p(G_2,\wg^q{{\Fg_1}}^\ast)\lra C_d^{p+1}(G_2,\wg^q{{\Fg_1}}^\ast)$ is the differentiable cohomology of $G_2$, with coefficients in $\wg^q{\Fg_1}^\ast$.

\noindent By \cite[Thm. 6.1]{HochMost62}, see also \cite{Dupo76,vanEst53}, one has the map
\begin{equation*}
\mathscr{F}^{p,q}:C^p_{\rm CE}(\Fg_2, \Fk, \wg^p{\Fg_1^\ast})\ra C_d^p(G_2,\wg^p{\Fg_1}^\ast)
\end{equation*}
of bicomplexes \cite{MoscRang11}, which  induces an isomorphism on the level of total cohomologies, as well as an isomorphism on the level of row cohomologies.
\end{proof}

\bibliographystyle{amsplain}
\bibliography{Rangipour-Sutlu-References}{}

\end{document}